\documentclass[reqno,11pt]{amsart}
\usepackage{bm,url,amsmath,amssymb,amsthm}
\usepackage{mathrsfs,stmaryrd}
\usepackage{enumerate}

\usepackage{graphicx,epic,eepic}

\usepackage[top=1in, bottom=1in, left=1in, right=1in]{geometry}
\usepackage{setspace}
\allowdisplaybreaks

\newcommand{\mcA}{\mathcal A}
\newcommand{\bA}{\mathbf A}
\newcommand{\mcB}{\mathcal B}
\newcommand{\mcC}{\mathcal C}
\newcommand{\mcD}{\mathcal D}
\newcommand{\mcE}{\mathcal E}
\newcommand{\mcH}{\mathcal H}
\newcommand{\mcK}{\mathcal K}
\newcommand{\N}{\mathbb N}
\newcommand{\Q}{\mathbb Q}
\newcommand{\R}{\mathbb R}
\newcommand{\T}{\mathbb T}
\newcommand{\Z}{\mathbb Z}
\newcommand{\FF}{\mathcal F}
\newcommand{\SL}{\mathrm{SL}}
\newcommand{\bsl}{\backslash}
\newcommand{\ol}{\overline}
\newcommand{\ul}{\underline}
\newcommand{\rar}{\rightarrow}
\newcommand{\mb}{\boldsymbol}
\newcommand{\diag}{\mathrm{diag}}
\newcommand{\Mod}[1]{\ (\mathrm{mod}\ #1)}
\DeclareMathOperator{\supp}{supp}

\newcommand{\FFA}{{\FF_\bA}}

\newcommand{\FFD}{{\FF_\Delta}}
\newcommand{\FFDl}{\FF_\Delta}
\newcommand{\TDn}{{\T_\Delta^n}}
\newcommand{\mfC}{\mathfrak C}
\newcommand{\EST}{\mathrm{EST}}

\newtheorem{theorem}{Theorem}
\newtheorem{corollary}{Corollary}

\theoremstyle{definition}
\newtheorem{remark}{Remark}
\newtheorem*{acknowledgements}{Acknowledgements}

\title[Equidistribution of Farey sequences on horospheres]{Equidistribution of Farey sequences on horospheres in covers of $\SL(n+1,\Z)\bsl\SL(n+1,\R)$ and applications}
\author{Byron Heersink}
\address{Department of Mathematics, The Ohio State University, Columbus, OH 43210}
\email{heersink.5@osu.edu}

\begin{document}

\begin{abstract}
We establish the limiting distribution of certain subsets of Farey sequences, i.e., sequences of primitive rational points, on expanding horospheres in covers $\Delta\bsl\SL(n+1,\R)$ of $\SL(n+1,\Z)\bsl\SL(n+1,\R)$, where $\Delta$ is a finite index subgroup of $\SL(n+1,\Z)$. These subsets can be obtained by projecting to the hyperplane $\{(x_1,\ldots,x_{n+1})\in\R^{n+1}:x_{n+1}=1\}$ sets of the form $\bA=\bigcup_{j=1}^J\mb{a}_j\Delta$, where for all $j$, $\mb{a}_j$ is a primitive lattice point in $\Z^{n+1}$. Our method involves applying the equidistribution of expanding horospheres in quotients of $\SL(n+1,\R)$ developed by Marklof and Str\"{o}mbergsson, and more precisely understanding how the full Farey sequence distributes in $\Delta\bsl\SL(n+1,\R)$ when embedded on expanding horospheres as done in previous work by Marklof. For each of the Farey sequence subsets, we extend the statistical results by Marklof regarding the full multidimensional Farey sequences, and solutions by Athreya and Ghosh to Diophantine approximation problems of Erd\H{o}s-Sz\"{u}sz-Tur\'an and Kesten. We also prove that Marklof's result on the asymptotic distribution of Frobenius numbers holds for sets of primitive lattice points of the form $\bA$.
\end{abstract}

\maketitle

\section{Introduction}

Let $\hat{\Z}^{n+1}=\{\mb{a}=(a_1,\ldots,a_{n+1})\in\Z^{n+1}:\gcd(a_1,\ldots,a_{n+1})=1\}$ be the set of primitive lattice points in $\Z^{n+1}$. The $n$-dimensional Farey sequence $\FF(Q)$ of level $Q\in\N$ is defined as the set of rational points $\mb{p}/q\in\R^n$ such that $(\mb{p},q)\in\hat{\Z}^{n+1}$ and $1\leq q\leq Q$. In \cite{M2}, Marklof proved certain limiting statistical properties of $\FF(Q)$ as $Q\rar\infty$ as a consequence of the equidistribution of the Farey points on expanding horospheres in a certain subspace of $\SL(n+1,\Z)\bsl\SL(n+1,\R)$, which he demonstrated in the course proving the asymptotic distribution of Frobenius numbers \cite{M1}. Fundamentally, these results follow from the equidistribution of expanding horospheres in quotients of $\SL(n+1,\R)$ by lattices obtained by Marklof and Str\"{o}mbergsson in the important work \cite{MS}. Recently, Athreya and Ghosh \cite{AG} extended and solved Diophantine approximation problems of Erd\H{o}s, Sz\"{u}sz, and Tur\'{a}n \cite{EST} and Kesten \cite{K} to, among other settings, Euclidean dimensions $n\geq2$ by finding limiting distributions of particular measures of subsets of $[0,1]^n$ which are well-approximated by certain elements in the Farey sequence.

In dimension $n=1$, Athreya and Cheung \cite{AC} provided a unified explanation for various statistical properties of Farey fractions, which were originally proven by analytic methods, by realizing the horocycle flow in $\SL(2,\R)/\SL(2,\Z)$ as a suspension flow over the BCZ map introduced by Boca, Cobeli, and Zaharescu \cite{BCZ} in their study of Farey fractions. The author's recent work \cite{H} used a process of Fisher and Schmidt \cite{FS} to lift the Poincar\'e section of Athreya and Cheung to obtain sections of the horocycle flow in covers $\SL(2,\R)/\Delta$ of $\SL(2,\R)/\SL(2,\Z)$, which in turn yielded results on the spacing statistics of the various subsets of Farey fractions related to those lifted sections.

In this paper, we extend this lifing method to higher dimensions in order to obtain results analogous to those of Marklof \cite{M2} and Athreya and Ghosh \cite{AG} for subsets of the multidimensional Farey sequences associated with finite index subgroups $\Delta\subseteq\SL(n+1,\Z)$ giving covers $\Delta\bsl\SL(n+1,\R)$ of $\SL(n+1,\Z)\bsl\SL(n+1,\R)$. Additionally, we show that the limiting distribution of Frobenius numbers established by Marklof \cite{M1} holds also for restricted sets of primitive lattice points given by orbits of $\Delta$. The method of proof mimics \cite{M1,M2} in utilizing the equidistribution of expanding horosphers and their Farey points. Specifically, we appeal to Marklof's extension of the Farey point equidistribution to general homogeneous spaces $\Gamma'\bsl\SL(n+1,\R)$ \cite{M3}. For $\Gamma'=\Delta$, we can locate the desired restricted subset of Farey points in appropriate sheets of the cover $\Delta\bsl\SL(n+1,\R)\rar\SL(n+1,\Z)\bsl\SL(n+1,\R)$, which then allows us to discern the equidistribution of the subset.

In Section \ref{Sec:Farey}, we establish notation and review the results of Marklof on the spacing statistics of the full Farey sequences \cite{M2} and the distribution of Frobenius numbers \cite{M1}, and the Diophantine results of Athreya and Ghosh \cite{AG}. We also formulate analogous results for Farey sequence subsets, and corresponding primitive lattice points, formed from the given subgroup $\Delta$. In Section \ref{Sec:Equid}, we formulate the appropriate variation of the equidistribution of horospheres and Farey points in $\Delta\bsl \SL(n+1,\R)$, following essentially from works of Marklof and Str\"{o}mbergsson \cite{M1,M3,MS}. In Section \ref{Sec:Proofs}, we prove our spacing statistics and Diophantine results for the Farey sequence subsets. Lastly, in Section \ref{Sec:Frob} we establish the limiting distribution of the Frobenius numbers for orbits of $\Delta$.

\section{Farey sequences and horospheres in $\SL(n+1,\R)$}\label{Sec:Farey}

Since $\FF(Q)$ is closed under addition by lattice points in $\Z^n$, we may view $\FF(Q)$ as a finite subset of $\T^n=\R^n/\Z^n$:
\[\FF(Q)=\left\{\frac{\mb{p}}{q}\in\T^n:(\mb{p},q)\in\hat{\Z}^{n+1},1\leq q\leq Q\right\}.\]
In \cite{M2}, Marklof considered the following statistical measures on $\FF(Q)$: For $k\in\Z_{\geq0}$ and subsets $\mcD\subseteq\T^n$ and $\mcA\subseteq\R^n$ that are bounded and have boundaries of Lebesgue measure zero and nonempty interiors, let
\begin{align}
P_Q(k,\mcD,\mcA)&=\frac{\lambda(\{\mb{x}\in\mcD:\#((\mb{x}+\sigma_Q^{-1/n}\mcA)\cap\FF(Q))=k\})}{\lambda(\mcD)}\text{ and}\label{Eq:Fstat1}\\
P_{0,Q}(k,\mcD,\mcA)&=\frac{\#\{\mb{r}\in\FF(Q)\cap\mcD:\#(\mb{r}+\sigma_Q^{-1/n}\mcA)\cap\FF(Q))=k\}}{\#(\FF(Q)\cap\mcD)}.\label{Eq:Fstat2}
\end{align}
Above and throughout this paper, $\lambda$ denotes the Lebesgue measure (by abuse of notation, on any torus or Euclidean space of any dimension) and
\begin{equation}\label{Eq:sigma}
\sigma_Q=\frac{Q^{n+1}}{(n+1)\zeta(n+1)}
\end{equation}
gives the asymptotic growth rate of $\#\FF(Q)$ as $Q\rar\infty$. The quantity $P_{0,Q}(k,\mcD,\mcA)$ is a higher dimensional analogue of the gap distribution in dimension $1$ and provides a measure for the statistical distribution of the location of the points in $\FF(Q)$ relative to one another; and the quantity $P_Q(k,\mcD,\mcA)$ measures the statistical distribution of the points $\FF(Q)$ relative to a random point in $\T^n$.

Let $G=\SL(n+1,\R)$ and $\Gamma=\SL(n+1,\Z)$, and similarly $G_0=\SL(n,\R)$ and $\Gamma_0=\SL(n,\Z)$. Then let $\mu$ denote the Haar measure on $G$ such that the $\mu$-measure of $\Gamma\bsl G$ is $1$; and view $\mu$ as a measure on other quotients $\Gamma'\bsl G$ by discrete groups $\Gamma'$ in the natural way. Similarly, let $\mu_0$ be the Haar measure on $G_0$ such that $\Gamma_0\bsl G_0$ is of $\mu_0$-measure $1$. Also define the subgroups $H\subseteq G$ and $\Gamma_H\subseteq\Gamma$ by
\begin{align*}
H&=\{M\in G:(\mb{0},1)M=(\mb{0},1)\}=\left\{\begin{pmatrix}A&\mb{b}^t\\\mb{0}&1\end{pmatrix}:A\in G_0,\mb{b}\in\R^n\right\}\text{ and}\\
\Gamma_H&=\Gamma\cap H=\left\{\begin{pmatrix}\gamma&\mb{m}^t\\\mb{0}&1\end{pmatrix}:\gamma\in\Gamma_0,\mb{m}\in\Z^n\right\}.
\end{align*}
Then let $\mu_H$ denote the Haar measure on $H$ such that the $\mu_H$-measure of $\Gamma_H\bsl H\cong\Gamma\bsl\Gamma H$ equals $1$. In \cite{M2}, $\FF(Q)$ was embedded in $\Gamma\bsl G$ as follows: Define the matrices
\[h(\mb{x})=\begin{pmatrix}I_n&\mb{0}^t\\
-\mb{x}&1\end{pmatrix}\quad\text{and}\quad
a(y)=\begin{pmatrix}y^{1/n}I_n&\mb{0}^t\\
\mb{0}&y^{-1}\end{pmatrix}\]
for $\mb{x}\in\R^n$ and $y\in(0,\infty)$. Then $h(\mb{x})$ parameterizes a horosphere which expands in $\Gamma\bsl G$ under right multiplication of $a(y)$ as $y\rar\infty$. Then a given Farey point $\mb{r}\in\FF(Q)$ of level $Q$ was associated with the element $\Gamma h(\mb{r})a(Q)\in\Gamma\bsl G$. The existence of the limit as $Q\rar\infty$ of \eqref{Eq:Fstat1} and \eqref{Eq:Fstat2} was then reduced by Marklof to showing the equidistribution of the horosphere $\{\Gamma h(\mb{x})a(Q):\mb{x}\in\T^n\}$ in $\Gamma\bsl G$, and the points $\{\Gamma h(\mb{r})a(Q):\mb{r}\in\FF(Q)\}$ in the subspace $\Gamma\bsl\Gamma H\{a(y):y\geq1\}$, respectively (see Section \ref{Sec:Proofs}). For ease of notation, denote $H_a=H\{a(y):y\geq1\}$.

We now fix a finite index subgroup $\Delta\subseteq\Gamma$, and let $\mu_\Delta$ be the Haar probability measure for $\Delta\bsl G$. We also fix elements $\mb{a}_1,\ldots,\mb{a}_J\in\hat{\Z}^{n+1}$. Then define the set of primitive lattice points
\begin{equation}\label{Eq:Aset}
\bA=\bigcup_{j=1}^J\bm{a}_j\Delta,
\end{equation}
for which we can define the set $\FFA(Q)$ consisting of the Farey points $\mb{p}/q$ such that $(\mb{p},q)\in\bA$ and $1\leq q\leq Q$. Note that since $\Delta$ is a finite index subgroup of $\SL(n+1,\Z)$, the set
\[\left\{\mb{x}\in\R^n:
\begin{pmatrix}I_n&\mb{0}^t\\\mb{x}&1\end{pmatrix}\in\Delta\right\}\]
is a sublattice $\Lambda_\Delta$ of rank $n$ in $\Z^n$. Then for $\mb{p}/q\in\FFA(Q)$ so that $(\mb{p},q)\in\mb{a}_j\Delta$ for some $j$, and $\mb{m}\in\Lambda_\Delta$, we have
\[(\mb{p}+q\mb{m},q)=(\mb{p},q)
\begin{pmatrix}I_n&\mb{0}^t\\\mb{m}&1\end{pmatrix}
\in\mb{a}_j\Delta\begin{pmatrix}I_n&\mb{0}^t\\\mb{m}&1\end{pmatrix}
=\mb{a}_j\Delta.\]
Thus $\FFA(Q)$ is closed under addition by elements in $\Lambda_\Delta$, and so we may view $\FFA(Q)$ as a subset of $\T_\Delta^n=\R^n/\Lambda_\Delta$:
\[\FFA(Q)=\left\{\frac{\mb{p}}{q}\in\T_\Delta^n:(\mb{p},q)\in\bA,1\leq q\leq Q\right\}.\]
For convenience, we define $\FF_\Delta(Q)$ to be the set of all Farey points of level $Q$, viewed as a subset of $\T_\Delta^n$, so that $\FFA(Q)\subseteq\FF_\Delta(Q)$. (Essentially, $\FF_\Delta(Q)$ is $[\Z^n:\Lambda_\Delta]$ copies of $\FF(Q)$.) Also, let $\lambda_\Delta$ denote the Lebesgue probability measure on $\T_\Delta^n$. The primary dynamical result in this paper is the equidistribution of the restricted set Farey points $\{\Delta h(\mb{r})a(Q):\mb{r}\in\FFA(Q)\}$ in a subspace of $\Delta\bsl G$. Our first application of this result, and the equidistribution of the entire horospheres in $\Delta\bsl G$, is the following:

\begin{theorem}\label{T:Stats}
The sequence $(\FFA(Q))_Q$ equidistributes in $\TDn$ with respect to the measure $\lambda_\Delta$. Also, for $k\in\Z_{\geq0}$ and bounded subsets $\mcD\subseteq\T_\Delta^n$ and $\mcA\subseteq\R^n$ with boundaries of measure zero and nonempty interiors, let
\begin{align}
P_Q^\bA(k,\mcD,\mcA)&=\frac{\lambda_\Delta(\{\mb{x}\in\mcD:\#((\mb{x}+(\#\FFA(Q))^{-1/n}\mcA)\cap\FFA(Q))=k\})}{\lambda_\Delta(\mcD)}\text{, and}\label{Eq:Astat1}\\
P_{0,Q}^\bA(k,\mcD,\mcA)&=\frac{\#\{\mb{r}\in\FFA(Q)\cap\mcD:\#(\mb{r}+(\#\FFA(Q))^{-1/n}\mcA)\cap\FFA(Q))=k\}}{\#(\FFA(Q)\cap\mcD)}.\label{Eq:Astat2}
\end{align}
Then the limit as $Q\rar\infty$ of both quantities exist and are independent of $\mcD$. Furthermore, we have the following evaluation of the limiting expected value of $k$ corresponding to the measures $P_Q^\bA(k,\mcD,\mcA)$:
\[\lim_{Q\rar\infty}\sum_{k=0}^\infty kP_Q^\bA(k,\mcD,\mcA)=\lambda_\Delta(\mcA).\]
\end{theorem}

The main noteworthy application of Theorem \ref{T:Stats} comes from letting $\Delta=\Gamma(m)$, where $m$ is a positive integer and $\Gamma(m)\subseteq\Gamma$ is the congruence subgroup
\[\Gamma(m)=\{M\in\Gamma:M\equiv I_{n+1}\Mod{m}\}.\]
For a given $\mb{a}\in\hat{\Z}^{n+1}$, $\mb{a}\Gamma(m)=\{\mb{b}\in\hat{\Z}^{n+1}:\mb{b}\equiv\mb{a}\Mod{m}\}$. Therefore, a set of the form $\bA$ in \eqref{Eq:Aset} in this situation is the set of points in $\hat{\Z}^{n+1}$ which are congruent modulo $m$ to a vector in $\{\mb{a}_1,\ldots,\mb{a}_J\}$. As a result, the set $\FFA(Q)$ consists of the Farey points $\mb{p}/q$ such that $(\mb{p},q)\equiv\mb{a}_j\Mod{m}$ for some $j$.

\begin{remark}
Equidistribution results in the spaces $\Gamma(m)\bsl G$ have also been used in \cite[Theorem 2.1]{MS} to understand the fine-scale statistics of the directions of the visible lattice points in $\Z^n$ with respect to an observer at a rational point in $\R^n$.
\end{remark}

In \cite{EST}, Erd\H{o}s, Sz\"usz, and Tur\'{a}n introduced the following Diophantine approximation problem: For constants $Q\in\N$, $A>0$, and $c>1$, let $\EST_{A,c,Q}$ be the function on $[0,1]$ such that for $x\in[0,1]$, $\EST_{A,c,Q}(x)$ is the number of solutions $p/q\in\Q$ satisfying
\[|qx - p|\leq\frac{A}{q},\qquad Q\leq q\leq cQ.\]
Then for fixed $A$ and $c$, determine the existence of the limit
\begin{equation}\label{ESTlimit}
\lim_{Q\rar\infty}\lambda(\{x\in[0,1]:\EST_{A,c,Q}(x)>0\}).
\end{equation}
Another Diophantine problem of Kesten \cite{K} is as follows: Define the function $K_{A,Q}$ on $[0,1]$ such that $K_{A,Q}(x)$ is the number of solutions $p/q\in\Q$ satisfying
\[|qx - p|\leq\frac{A}{Q},\qquad 1\leq q\leq Q.\]
Kesten's problem is to determine the existence of the limit
\begin{equation}\label{Klimit}
\lim_{Q\rar\infty}\lambda(\{x\in[0,1]:K_{A,Q}(x)>0\}).
\end{equation}

In the original paper \cite{EST}, Erd\H{o}s, Sz\"{u}sz, and Tur\'{a}n showed that the limit \eqref{ESTlimit} exists when $A\leq c/(1+c^2)$. Later, while resolving the limit \eqref{Klimit}, Kesten \cite{K} extended the existence of \eqref{ESTlimit} for when $Ac\leq1$. The limit was then shown to exist in all cases by Kesten and S\'{o}s \cite{KS}. Explicit formulas for the limit were obtained much later by Xiong and Zaharescu \cite{XZ} and Boca \cite{B}.

In their recent work \cite{AG}, Athreya and Ghosh generalized the Erd\H{o}s-Sz\"{u}sz-Tur\'{a}n and Kesten problems to higher dimensions, in addition to other settings such as translation surfaces. They proved the existence of the limiting distribution of the functions $\EST_{A,c,Q}^n$ and $K_{A,Q}^n$ on $[0,1]^n$, analogous to the corresponding functions above, defined so that for $\mb{x}\in[0,1]^n$, $\EST_{A,c,Q}^n(\mb{x})$ is the number of solutions $(\mb{p},q)\in\hat{\Z}^{n+1}$ satisfying
\begin{equation}\label{Eq:ESTIneq}
\|q\mb{x}-\mb{p}\|\leq Aq^{-1/n},\qquad Q\leq q\leq cQ,
\end{equation}
where $\|\cdot\|$ denotes the Euclidean norm on $\R^n$; and $K_{A,Q}^n(\mb{x})$ is the number of solutions satisfying
\begin{equation}\label{Eq:KIneq}
\|q\mb{x}-\mb{p}\|\leq AQ^{-1/n},\qquad 1\leq q\leq Q.
\end{equation}
Specifically, they showed that for each $k\in\Z_{\geq0}$, the limits 
\[\lim_{Q\rar\infty}\lambda(\{\mb{x}\in[0,1]^n:\EST_{A,c,Q}^n(\mb{x})=k\})\quad\text{and}\quad\lim_{Q\rar\infty}\lambda(\{\mb{x}\in[0,1]^n:K_{A,Q}^n(\mb{x})=k\})\]
exist and can be expressed in terms of the Haar measure of certain subsets of the space of unimodular lattices $G/\Gamma$.

In this paper, we extend the results of Athreya and Ghosh to our setting, restricting the set of points $(\mb{p},q)$ in the inequalities \eqref{Eq:ESTIneq} and \eqref{Eq:KIneq} to $\bA$. That is, we define the new functions $\EST_{A,c,Q}^\bA$ and $K_{A,Q}^\bA$ on $\TDn$ so that for $\mb{x}\in\TDn$, $\EST_{A,c,Q}^\bA(\mb{x})$ is the number of solutions $(\mb{p},q)\in\bA$ satisfying \eqref{Eq:ESTIneq}, and $K_{A,Q}^\bA(\mb{x})$ is the number of solutions satisfying \eqref{Eq:KIneq}. Our second main result is the following:

\begin{theorem}\label{T:KEST}
For fixed $A>0$ and $c>1$, the limiting distributions of the functions $\EST_{A,c,Q}^\bA$ and $K_{A,Q}^\bA$ as $Q\rar\infty$ exist. More specifically, for each $k\in\Z_{\geq0}$ and $\mcD\subseteq\TDn$ with boundary of measure zero and nonempty interior, the limits
\[\lim_{Q\rar\infty}\frac{\lambda_\Delta(\{\mb{x}\in\mcD:\EST_{A,c,Q}^\bA(\mb{x})=k\})}{\lambda_\Delta(\mcD)}\quad\text{and}\quad\lim_{Q\rar\infty}\frac{\lambda_\Delta(\{\mb{x}\in\mcD:K_{A,Q}^\bA(\mb{x})=k\})}{\lambda_\Delta(\mcD)}\]
exist and are independent of $\mcD$. Also, the limiting expected values of $\EST_{A,c,Q}^\bA$ and $K_{A,Q}^\bA$ as $Q\rar\infty$ exist.
\end{theorem}

Analogously to Theorem \ref{T:Stats}, letting $\Delta=\Gamma(m)$ allows us to obtain limiting distributions for $\EST$ and $K$ functions corresponding to lattice points satisfying congruence conditions modulo $m$.

The original motivation for studying the distribution of Farey sequences on horospheres was in the study of Frobenius numbers. For a given $\mb{a}$ in the set $\hat{\Z}_{\geq2}^{n+1}$ of primitive integer lattice points with coordinates at least $2$, the Frobenius number $F(\mb{a})$ of $\mb{a}$ is defined as the largest natural number which cannot be represented as a non-negative integer combination of the coordinates of $\mb{a}$, that is,
\[F(\mb{a})=\max\left(\N\bsl\{\mb{m}\cdot\mb{a}:\mb{m}=(m_1,\ldots,m_{n+1})\in\Z_{\geq0}^{n+1}\}\right).\]
For $n=1$, the equality $F(\mb{a})=a_1a_2-a_1-a_2$ holds, while for higher dimensions no explicit formula is known. However, in \cite{M1}, Marklof determined the limit distribution of Frobenius numbers for $n\geq2$ by relating the values of Frobenius numbers $F(\mb{a})$ to the location of points in the Farey sequence $\FF(Q)$ when embedded in $\Gamma\bsl G$. Marklof proved \cite[Theorem 1]{M1} that there is a continuous non-increasing function $\Psi_{n+1}:\R_{\geq0}\rar\R_{\geq0}$ such that for bounded $\mcD\subseteq\R_{\geq0}^{n+1}$ with boundary of measure zero and $R\geq0$,
\[\lim_{T\rar\infty}\frac{1}{T^{n+1}}\#\left\{\mb{a}\in\hat{\Z}_{\geq2}^{n+1}\cap T\mcD:\frac{F(\mb{a})}{(a_1\cdots a_{n+1})^{1/n}}>R\right\}=\frac{\lambda(\mcD)}{\zeta(n+1)}\Psi_{n+1}(R).\]
The function $\Psi_{n+1}$ can be expressed in terms of the covering radius of the simplex
\[\delta^{(n)}=\{\mb{x}\in\R_{\geq0}^n:\mb{x}\cdot\mb{e}\leq1\},\qquad\mb{e}=(1,1,\ldots,1),\]
with respect to lattices in $\Gamma_0\bsl G_0$. The covering radius of $\delta^{(n)}$ with respect to $\Gamma_0A\in\Gamma_0\bsl G_0$ is the quantity $\rho(\Gamma_0A)$ defined by
\begin{equation}\label{Eq:covrad}
\rho(\Gamma_0A)=\inf\{\rho'>0:\Z^nA+\rho'\delta^{(n)}=\R^n\}.
\end{equation}
We then have \cite[Theorem 2]{M1}
\[\Psi_{n+1}(R)=\mu_0(\{A\in\Gamma_0\bsl G_0:\rho(A)>R\}).\]
The last main result of this paper is to prove that the distribution of Frobenius numbers according to $\Psi_{n+1}$ continues to hold when restricting the lattice points one considers to $\bA$.

\begin{theorem}\label{T:Fr}
Assume $n\geq2$ and let $\mcD\subseteq\R_{\geq0}^{n+1}$ be bounded with boundary of measure zero. Then there exists a positive integer $i_\bA\leq[\Gamma:\Delta]$, depending on $\bA$, such that
\begin{equation}\label{Eq:Fr}
\lim_{T\rar\infty}\frac{1}{T^{n+1}}\#\left\{\mb{a}\in\hat{\Z}_{\geq2}^{n+1}\cap\bA\cap T\mcD:\frac{F(\mb{a})}{(a_1\cdots a_{n+1})^{1/n}}>R\right\}=\frac{i_\bA\lambda(\mcD)}{[\Gamma:\Delta]\zeta(n+1)}\Psi_{n+1}(R).
\end{equation}
\end{theorem}

We record the special case of Theorems \ref{T:Stats}, \ref{T:KEST}, and \ref{T:Fr} for $\Delta=\Gamma(m)$, in which case $\bA$ is the set of points in $\hat{\Z}^{n+1}$ satisfying certain congruence conditions modulo $m$, in the following corollary.

\begin{corollary}\label{C:1}
Fix $\mb{a}_1,\ldots,\mb{a}_J\in\hat{\Z}^{n+1}$ and let $\bA$ be the set of points $\mb{a}\in\hat{\Z}^{n+1}$ such that $\mb{a}\equiv\mb{a}_j\Mod{m}$ for some $j$. Then let $\FFA(Q)\subseteq(\R/m\Z)^n$ be the set of Farey points $\mb{p}/q$ such that $(\mb{p},q)\in\bA$. 
\begin{enumerate}[(a)]
\item The sequence $(\FFA(Q))_Q$ equidistributes with respect to the Lebesgue probability measure $\lambda_{\Gamma(m)}$ on $(\R/m\Z)^n$; and the limit as $Q\rar\infty$ of the spacing statistics $P_Q^\bA$ and $P_{0,Q}^\bA$ defined by \eqref{Eq:Astat1} and \eqref{Eq:Astat2}, respectively, exist for all $k\in\Z_{\geq0}$ and bounded subsets $\mcD\subseteq(\R/m\Z)^n$ and $\mcA\subseteq\R^n$ with boundaries of measure zero and nonempty interiors. Also, the expected value of $k$ corresponding to the measure $P_Q^\bA$ exists and equals $\lambda_{\Gamma(m)}(\mcA)$.\label{C:1a}
\item For $\mb{x}\in(\R/m\Z)^n$, define $\EST_{A,c,Q}^\bA(\mb{x})$ to be the number of solutions $(\mb{p},q)\in\bA$ satisfying \eqref{Eq:ESTIneq}, and $K_{A,Q}^\bA(\mb{x})$ the number of solutions satisfying \eqref{Eq:KIneq}. Then for $A$, $c$, and $\mcD$ as in Theorem \ref{T:KEST}, the limiting distributions and expectations for the functions $\EST_{A,c,Q}^\bA$ and $K_{A,Q}^\bA$ as $Q\rar\infty$ exist, and do not depend on the restricting subset $\mcD$.\label{C:1b}
\item Assume that $n\geq2$ and $\mcD\subseteq\R_{\geq0}^{n+1}$ is bounded and has boundary of measure zero. Then \eqref{Eq:Fr} holds for $\Delta=\Gamma(m)$ and some positive integer $i_\bA\leq[\Gamma:\Gamma(m)]$.
\end{enumerate}
\end{corollary}

\begin{remark}
For $n\geq2$, Theorems \ref{T:Stats}, \ref{T:KEST}, and \ref{T:Fr} can be essentially reduced to Corollary \ref{C:1} by the congruence subgroup property of $\SL(n+1,\Z)$ \cite{BLS,Me}. However, this is not the case for $n=1$ in Theorems \ref{T:Stats} and \ref{T:KEST}.
\end{remark}

\begin{remark}
Marklof's result \cite{M1} was improved in \cite{L}, where Li obtained an effective equidistribution result for the Farey points $\{\Gamma h(\mb{r})a(Q):\mb{r}\in\FF(Q)\}$, as well as the horosphere $\{\Gamma h(\mb{x})a(Q):\mb{x}\in\T^n\}$, which in turn allowed him to obtain an error term for the limit distribution of Frobenius numbers. Einsiedler et al.\ \cite{EMSS} made another advancement by proving the equidistribution of the Farey points in $\Gamma\bsl G$ corresponding to the elements $\mb{p}/q\in\FF(Q)$ such that $q=Q$. (They additionally obtained analogous equidistribution results for horospheres in $\SL(n+1,\R)$ taking up more than one matrix row.) This result was then made effective in dimension $n=2$ by Lee and Marklof \cite{LM}.
\end{remark}

\section{Equidistribution in $\Delta\bsl G$}\label{Sec:Equid}

We now set out to prove Theorems \ref{T:Stats}, \ref{T:KEST}, and \ref{T:Fr}, appealing to the equidistribution results of Marklof and Str\"{o}mbergsson \cite{M1,M2,M3,MS}. We shall derive the limit of the measure $P_Q^\bA$, and the limiting distributions of $\EST_{A,c,Q}^\bA$ and $K_{A,Q}^\bA$, from the equidistribution of expanding horospheres in $\Delta\bsl G$. We specifically use the following adaptation of \cite[Theorem 5.8]{MS} (see also \cite[Theorem 1]{M2}).

\begin{theorem}\label{T:Equid1}
Let $f:\TDn\times\Delta\bsl G\rar\R$ be bounded and continuous. Then
\[\lim_{Q\rar\infty}\int_{\T_\Delta^n}f(\mb{x},h(\mb{x})a(Q))\,d\lambda_\Delta(\mb{x})=\int_{\T_\Delta^n\times\Delta\bsl G}f(\mb{x},M)\,d\lambda_\Delta(\mb{x})\,d\mu_\Delta(M).\]
\end{theorem}

This result follows from the mixing of the diagonal subgroup $\{a(y):y>0\}$ on $\Delta\bsl G$ via the arguments in \cite{EM}; we omit the proof. 

Next, let $\pi_\Delta:\Delta\bsl G\rar\Gamma\bsl G$ be the canonical projection. By \cite[Theorem 6]{M1}, the points $\{\Gamma h(\mb{r})a(Q):\mb{r}\in\FF(Q)\}$ equidistribute in the subspace $\Omega=\Gamma\bsl\Gamma H_a$ of $\Gamma\bsl G$ with respect to the measure $d\mu_\Omega(Ma(y))=d\mu_H(M)(n+1)y^{-(n+2)}d\lambda(y)$. In our setting, we seek to lift $\Omega$ via $\pi_\Delta$ to get the subspace $\Omega_\Delta=\pi_\Delta^{-1}(\Omega)=\Delta\bsl\Gamma H_a$ of $\Delta\bsl G$, in which to obtain the analogous equidistribution of $\{\Delta h(\mb{r})a(Q):\mb{r}\in\FF_\Delta(Q)\}$ with respect to the probability measure $\mu_{\Omega_\Delta}$ obtained as the normalized pullback of $\mu_\Omega$ with respect to $\pi_\Delta$. This is in essence the content of \cite[Theorem 2(A)]{M3}. We refine this result to find the equidistribution of the restricted set $\{\Delta h(\mb{r})a(Q):\mb{r}\in\FFA(Q)\}$, which we then use to establish the equidistribution of $(\FFA(Q))_Q$ in $\T_\Delta^n$, the existence of the limit of the measure $P_{0,Q}^\bA$, and the limiting distribution of the Frobenius numbers of $\bA$.

To obtain this result, we utilize the fact that the full set of Farey points naturally partition themselves into different sheets in the cover $\pi_\Delta|_{\Omega_\Delta}:\Omega_\Delta\rar\Omega$. We can therefore extract subsets of Farey points based on the particular sheets in which we are interested. The sheets corresponding to $(\FFA(Q))_Q$ are determined as follows: Define the subset $\bA^*\subseteq\Gamma$ to be the set of all $\gamma\in\Gamma$ for which there exists $\mb{a}\in\bA$ such that $\mb{a}\gamma=(\mb{0},1)$. Since $\bA$ is closed under right multiplication by $\Delta$, $\bA^*$ is closed under left multiplication by $\Delta$, and hence is a union of cosets in $\Delta\bsl\Gamma$. We shall therefore view $\bA^*$ as a subset of $\Delta\bsl\Gamma$.

Note that for $\Delta\gamma\in\bA^*$ so that $\mb{a}\gamma=(\mb{0},1)$ for some $\mb{a}\in\bA$, we have $\mb{a}\gamma\gamma'=(\mb{0},1)$ for any $\gamma'\in\Gamma_H$, implying that $\Delta\gamma\gamma'\subseteq\bA^*$. Thus $\bA^*$ is closed under right multiplication by elements in $\Gamma_H$. In fact, for any fixed $\gamma_1,\ldots,\gamma_j\in\Gamma$ such that $\mb{a}_j\gamma_j=(\mb{0},1)$, $\bA^*$ is the union of the orbits of the elements $\Delta\gamma_1,\ldots,\Delta\gamma_J\subseteq\Delta\bsl\Gamma$ under the action by right multiplication of $\Gamma_H$. Indeed, if $\Delta\gamma\in\bA^*$ so that $\mb{a}_j\delta\gamma=(\mb{0},1)$ for some $j\in\{1,\ldots,J\}$ and $\delta\in\Delta$, then $(\mb{0},1)(\delta\gamma)^{-1}\gamma_j=\mb{a}_j\gamma_j=(\mb{0},1)$. Thus $(\delta\gamma)^{-1}\gamma_j\in\Gamma_H$, implying that $\gamma_j^{-1}\delta\gamma=((\delta\gamma)^{-1}\gamma_j)^{-1}\in\Gamma_H$, implying that $\Delta\gamma$ is in the $\Gamma_H$-orbit of $\Delta\gamma_j$.

We show below that $\Omega_\bA=\Delta\bsl\bA^*H_a$ is the union of the appropriate sheets corresponding to $(\FFA(Q))_Q$. We can now formulate the main equidistribution result we need to prove Theorems \ref{T:Stats} and \ref{T:Fr}.

\begin{theorem}\label{T:Equid2}
Let $f:\T_\Delta^n\times\Omega_\Delta\rar\R$ be bounded and continuous. Then
\begin{enumerate}[(a)]
\item $\displaystyle{\lim_{Q\rar\infty}\frac{1}{\#\FF_\Delta(Q)}\sum_{\mb{r}\in\FF_\Delta(Q)}f(\mb{r},\Delta h(\mb{r})a(Q))=\int_{\T_\Delta^n\times\Omega_\Delta}f(\mb{x},M)\,d\lambda_\Delta(\mb{x})\,d\mu_{\Omega_\Delta}(M)}$,\label{T:Equid2pt1}
\item $\displaystyle{\lim_{Q\rar\infty}\frac{1}{\#\FF_\Delta(Q)}\sum_{\mb{r}\in\FF_\bA(Q)}f(\mb{r},\Delta h(\mb{r})a(Q))=\int_{\T_\Delta^n\times\Omega_\bA}f(\mb{x},M)\,d\lambda_\Delta(\mb{x})\,d\mu_{\Omega_\Delta}(M)}$.\label{T:Equid2pt2}
\end{enumerate}
\end{theorem}

Part \eqref{T:Equid2pt1} is essentially a corollary of \cite[Theorem 2(A)]{M3}. We nevertheless provide a detailed proof for our particular situation, closely following \cite[Theorem 6]{M1} and \cite[Theorem 2(A)]{M3}. We then obtain part \eqref{T:Equid2pt2} as a straightforward consequence, locating in $\Omega_\Delta$ the Farey points corresponding to the restricted set $\FFA(Q)$ as described above.

\begin{proof}[Proof of Theorem \ref{T:Equid2}]

\noindent\textbf{Part \eqref{T:Equid2pt1}.} 
To begin, we first note that it is elementary to show that $\Omega_\Delta$ is a closed subset of $\Delta\bsl G$. (In the proof of part \eqref{T:Equid2pt2}, we show that a set of the form $\Delta\bsl\Delta\gamma H_a$, with $\gamma\in\Gamma$, is closed in $\Delta\bsl G$.) Thus, since $\Delta\bsl G$ is a metric space, any bounded and continuous function on $\Omega_\Delta$ can be extended to a bounded continuous function on $\Delta\bsl G$. So we can assume without loss of generality that $f$ is a bounded continuous function on $\Delta\bsl G$. By standard approximation arguments, we can furthermore assume that $f$ has compact support, and hence is uniformly continuous. Let $\mcC\subseteq G$ be compact such that $\supp f\subseteq\T_\Delta^n\times\Delta\bsl\Delta\mcC$. Also, let $d:G\times G\rar\R_{\geq0}$ be a left invariant Riemannian metric on $G$ such that $d(h(\mb{x}),h(\mb{x}'))\leq\|\mb{x}-\mb{x}'\|$. We also let $d$ act as a metric on quotients of $G$ by discrete subgroups in the obvious manner. Note that by the uniform continuity of $f$, for a given $\delta>0$ there exists $\epsilon>0$ such that $|f(\mb{x},M)-f(\mb{x}',M')|\leq\delta$
whenever $\mb{x},\mb{x}'\in\T_\Delta^n$ and $M,M'\in\Delta\bsl G$ such that $\|\mb{x}-\mb{x}'\|<\epsilon$ and $d(M,M')<\epsilon$.

The basic plan of the proof of part \eqref{T:Equid2pt1} is as follows. For $\theta\in(0,1)$, let $\FF_\Delta^\theta(Q)=\{\mb{p}/q\in\FF_\Delta(Q):\theta Q\leq q\leq Q\}$. For a fixed $\theta$ and $\epsilon>0$, we use Theorem \ref{T:Equid1} to examine the value of the integral of $f(\mb{x},\Delta h(\mb{x})a(Q))$ over the points $\mb{x}\in\T_\Delta^n$ within $\epsilon/Q^{(n+1)/n}$ of an element in $\FF_\Delta^\theta(Q)$. Letting $\epsilon\rar0$ then establishes the equidistribution of the sequence Farey points corresponding to $(\FF_\Delta^\theta(Q))_Q$. The final step of the proof is taking the limit as $\theta\rar0$.

So let $\theta\in(0,1)$ and $\epsilon>0$, and define the set
\[\FF_\Delta^{\theta,\epsilon}(Q)=\bigcup_{\mb{r}\in\FF_\Delta^\theta(Q)}\left\{\mb{x}\in\T_\Delta^n:\|\mb{x}-\mb{r}\|\leq\frac{\epsilon}{Q^{(n+1)/n}}\right\}.\]
For $\mb{x}\in\R^{n+1}$ and $(\mb{p},q)\in\hat{\Z}^{n+1}$, we have $\|\mb{x}-\mb{p}/q\|\leq\epsilon/Q^{(n+1)/n}$ if and only if $(\mb{p},q)h(\mb{x})a(Q)$ is in the set
\[\mfC_{\theta,\epsilon}=\left\{(y_1,\ldots,y_{n+1})\in\R^{n+1}:\|(y_1,\ldots,y_n)\|\leq\epsilon y_{n+1},\theta\leq y_{n+1}\leq 1\right\}.\]
It is therefore straightforward to see that
\[\FFDl^{\theta,\epsilon}(Q)=\left\{\mb{x}\in\TDn:\hat{\Z}^{n+1}h(\mb{x})a(Q)\cap\mfC_{\theta,\epsilon}\neq\emptyset\right\}.\]
Let $\chi_\epsilon:G\rar\R$ be the characteristic function of the set
$\mcH_\epsilon=\{M\in G:\hat{\Z}^{n+1}M\cap\mfC_{\theta,\epsilon}\neq\emptyset\}$. Since $\mcH_\epsilon$ is closed under left multiplication by $\Gamma$, we may view $\chi_\epsilon$ as a function on $\Gamma\bsl G$ as well as $\Delta\bsl G$. Also, since $\mfC_{\theta,\epsilon}$ has boundary of Lebesgue measure zero, $\chi_\epsilon$ is continuous outside of a set of $\mu$-measure zero. By Theorem \ref{T:Equid1}, we have
\begin{align}
\lim_{Q\rar\infty}\int_{\FF_\Delta^{\theta,\epsilon}(Q)}f(\mb{x},\Delta h(\mb{x})a(Q))\,d\lambda_\Delta(\mb{x})&=\lim_{Q\rar\infty}\int_{\TDn}f(\mb{x},\Delta h(\mb{x})a(Q))\chi_\epsilon(\Delta h(\mb{x})a(Q))\,d\lambda_\Delta(\mb{x})\notag\\
&=\int_{\TDn\times\Delta\bsl G}f(\mb{x},M)\chi_\epsilon(M)\,d\lambda_\Delta(\mb{x})\,d\mu_\Delta(M)\notag\\
&=\int_{\TDn\times\Gamma\bsl G}\bar{f}(\mb{x},M)\chi_\epsilon(M)\,d\lambda_\Delta(\mb{x})\,d\mu(M),\label{Eq:int1}
\end{align}
where $\bar{f}:\TDn\times\Gamma\bsl G\rar\R$ is defined by
\[\bar{f}(\mb{x},\Gamma M)=\frac{1}{[\Gamma:\Delta]}\sum_{\Delta\gamma\in\Delta\bsl\Gamma}f(\mb{x},\Delta\gamma M).\]

Next, we wish to replace $\chi_\epsilon(M)$ in the product $\bar{f}(\mb{x},M)\chi_\epsilon(M)$ with the sum of characteristic functions of the sets $\mcH_\epsilon(\mb{a})=\{M\in G:\mb{a}M\in\mfC_{\theta,\epsilon}\}$, where $\mb{a}\in\hat{\Z}^{n+1}$. To do so, we must ensure that the sets $\mcH(\mb{a})$ are pairwise disjoint over the support of $\bar{f}$ when viewed as a function on $G$, which is $\Gamma\mcC$. It was proven in \cite{M2,M3} that this is always possible as long as $\epsilon$ is chosen to be small enough. We provide the following simpler proof of this fact.

Suppose on the contrary that for every $\epsilon>0$, the sets $\{\mcH_\epsilon(\mb{a}):\mb{a}\in\hat{\Z}^{n+1}\}$ are not pairwise disjoint over $\Gamma\mcC$. Then for every $j\in\N$, there exist distinct $\mb{a}_j,\mb{b}_j\in\hat{\Z}^{n+1}$ such that $\mcH_{1/j}(\mb{a}_j)\cap\mcH_{1/j}(\mb{b}_j)\cap\Gamma\mcC$; thus there exist $M_j\in G$, $\gamma_j\in\Gamma$, and $C_j\in\mcC$ such that $\mb{a}_jM_j\in\mfC_{\theta,1/j}$, $\mb{b}_jM_j\in\mfC_{\theta,1/j}$, and $M_j=\gamma_jC_j$. Then, since $\mcC$ is compact, we may assume by taking an appropriate subsequence that $(C_j)_j$ converges to an element $C\in\mcC$. Similarly, since $\mb{a}_jM_j,\mb{b}_jM_j\in\mfC_{\theta,1/j}$ for all $j\in\N$, we may take another subsequence so that
\[\lim_{j\rar\infty}\mb{a}_jM_j=\mb{a}\quad\text{and}\quad\lim_{j\rar\infty}\mb{b}_jM_j=\mb{b},\]
where $\mb{a},\mb{b}\in\{(\mb{0},y_{n+1})\in\R^{n+1}:\theta\leq y_{n+1}\leq1\}$. Since $M_j=\gamma_jC_j$ and $\lim_{j\rar\infty}C_j=C$, we have
\[\lim_{j\rar\infty}\mb{a}_j\gamma_j=\mb{a}C^{-1}\quad\text{and}\quad\lim_{j\rar\infty}\mb{b}_j\gamma_j=\mb{b}C^{-1}.\]
Since $\bm{a}_j\gamma_j,\mb{b}_j\gamma_j\in\hat{\Z}^{n+1}$ and $\hat{\Z}^{n+1}$ is discrete, there exists $N\in\N$ such that if $j\geq N$, then $\mb{a}_j\gamma_j=\mb{a}C^{-1}$ and $\mb{b}_j\gamma_j=\mb{b}C^{-1}$; hence $\mb{a}C^{-1},\mb{b}C^{-1}\in\hat{\Z}^{n+1}$. However, since $\mb{a}$ and $\mb{b}$ are positive multiples of $(\mb{0},1)$, $\mb{a}C^{-1}$ and $\mb{b}C^{-1}$ are postive multiples of the last row of $C^{-1}$, and thus of each other. This is possible only if $\mb{a}C^{-1}=\mb{b}C^{-1}$. So for $j\geq N$, $\mb{a}_j\gamma_j=\mb{a}C^{-1}=\mb{b}C^{-1}=\mb{b}_j\gamma_j$. This yields $\mb{a}_j=\mb{b}_j$, which contradicts our choice of $\mb{a}_j$ and $\mb{b}_j$ to be distinct.

Thus there exists $\epsilon_0>0$ such that if $\epsilon\in(0,\epsilon_0]$, which we henceforth assume, then the sets $\{\mcH_\epsilon(\mb{a}):\mb{a}\in\hat{\Z}^{n+1}\}$ are disjoint in $\Gamma\mcC$. Let $\chi_\epsilon^1:G\rar\R$ be the characteristic function of $\mcH_\epsilon^1=\mcH_\epsilon((\mb{0},1))$. Note that the map $\Gamma_H\gamma\mapsto(\mb{0},1)\gamma$ defines a bijection from $\Gamma_H\bsl\Gamma$ to $\hat{\Z}^{n+1}$, and the characteristic function of $\mcH_\epsilon^1((\mb{0},1)\gamma)$ is $M\mapsto\chi_\epsilon^1(\gamma M)$. These facts imply that
\[\chi_\epsilon(\Gamma M)=\sum_{\gamma\in\Gamma_H\bsl\Gamma}\chi_\epsilon^1(\gamma M)\quad\text{for $\Gamma M\in\Gamma\mcC$}.\]
Therefore, \eqref{Eq:int1} equals
\begin{align*}
\int_{\TDn\times\Gamma\bsl G}\bar{f}(\mb{x},M)\sum_{\gamma\in\Gamma_H\bsl\Gamma}\chi_\epsilon^1(\gamma M)&\,d\lambda_\Delta(\mb{x})\,d\mu(M)=\int_{\TDn\times\Gamma_H\bsl G}\bar{f}(\mb{x},M)\chi_\epsilon^1(M)\,d\lambda_\Delta(\mb{x})\,d\mu(M)\\
&=\int_{\TDn\times\Gamma_H\bsl\mcH_\epsilon^1}\bar{f}(\mb{x},M)\,d\lambda_\Delta(\mb{x})\,d\mu(M)\\
&=\frac{1}{\zeta(n+1)}\int\limits_{\TDn\times\Gamma_H\bsl H\times\mfC_{\theta,\epsilon}}\bar{f}(\mb{x},M M_{\mb{y}})\,d\lambda_\Delta(\mb{x})\,d\mu_H(M)d\lambda(\mb{y}).
\end{align*}
For the last step above, we parameterize $\Gamma_H\bsl\mcH_\epsilon^1$ by the coordinates $(M,\mb{y})\mapsto MM_{\mb{y}}$, where $(M,\mb{y})\in\Gamma_H\bsl H\times\mfC_{\theta,\epsilon}$ and 
\[M_{\mb{y}}=
\begin{pmatrix}y_{n+1}^{-1/n}I_n&\mb{0}^t\\\mb{y}'&y_{n+1}\end{pmatrix},\quad\mb{y}=(\mb{y}',y_{n+1}).\]
In these coordinates, we have $d\mu=\zeta(n+1)^{-1}\,d\mu_H(M)\,d\lambda(\mb{y})$.

So, to summarize thus far, we have
\begin{align}
&\lim_{Q\rar\infty}\int_{\FF_\Delta^{\theta,\epsilon}(Q)}f(\mb{x},\Delta h(\mb{x})a(Q))\,d\lambda_\Delta(\mb{x})\label{Eq:Fint1}\\
&\qquad\qquad\qquad\qquad=\frac{1}{\zeta(n+1)}\int\limits_{\TDn\times\Gamma_H\bsl H\times\mfC_{\theta,\epsilon}}\bar{f}(\mb{x},M M_{\mb{y}})\,d\lambda_\Delta(\mb{x})\,d\mu_H(M)d\lambda(\mb{y}).\label{Eq:int2}
\end{align}
We now divide both sides by the Lebesgue measure $\lambda(\mcB_\epsilon^n)=\epsilon^n\lambda(\mcB_1^n)$ of a ball $\mcB_\epsilon^n$ in $\R^n$ of radius $\epsilon$, and examine both sides as $\epsilon\rar0$. We begin with \eqref{Eq:int2}. By the uniform continuity of $f$, and hence of $\bar{f}$, for a given $\delta>0$ there exists $\epsilon_\delta>0$ such that $|\bar{f}(\mb{x},M)-\bar{f}(\mb{x}',M')|\leq\delta$ whenever $\|\mb{x}-\mb{x}'\|\leq\epsilon_\delta$ and $d(M,M')\leq\epsilon_\delta$. (Assume without loss of generality that $\epsilon_\delta\rar0$ as $\delta\rar0$, and that all $\epsilon_\delta$ are less than the constant $\epsilon_0$ found above.) For $\mb{y}=(\mb{y}',y_{n+1})\in\mfC_{\theta,\epsilon_\delta}$, we have
\[d(M_{\mb{y}},a(y_{n+1}^{-1}))=d(a(y_{n+1}^{-1})h(-y_{n+1}^{-1}\mb{y}'),a(y_{n+1}^{-1}))=d(h(-y_{n+1}^{-1}\mb{y}'),I_{n+1})\leq y_{n+1}^{-1}\|\mb{y}'\|\leq\epsilon_\delta.\]
Thus, $|\bar{f}(\mb{x},MM_{\mb{y}})-\bar{f}(\mb{x},Ma(y_{n+1}^{-1})|\leq\delta$ for all $\mb{x}\in\TDn$, $M\in\Gamma_H\bsl H$, and $\mb{y}\in\mfC_{\theta,\epsilon_\delta}$, implying that
\begin{align*}
&\frac{1}{\lambda(\mcB_{\epsilon_\delta})\zeta(n+1)}\int\limits_{\TDn\times\Gamma_H\bsl H\times\mfC_{\theta,\epsilon_\delta}}|\bar{f}(\mb{x},MM_{\mb{y}})-\bar{f}(\mb{x},Ma(y_{n+1}^{-1}))|\,d\lambda_\Delta(\mb{x})\,d\mu_H(M)\,d\lambda(\mb{y})\\
&\qquad\qquad\qquad\leq\frac{\delta\lambda(\mfC_{\theta,\epsilon_\delta})}{\zeta(n+1)\lambda(\mcB_\epsilon)}\leq\frac{\delta}{\zeta(n+1)}.
\end{align*}
Also, we have
\begin{align}
&\frac{1}{\lambda(\mcB_{\epsilon_\delta})\zeta(n+1)}\int\limits_{\TDn\times\Gamma_H\bsl H\times\mfC_{\theta,\epsilon_\delta}}\bar{f}(\mb{x},Ma(y_{n+1}^{-1}))\,d\lambda_\Delta(\mb{x})\,d\mu_H(M)\,d\lambda(\mb{y})\notag\\
&=\frac{1}{\zeta(n+1)}\int_\theta^1\int_{\TDn\times\Gamma_H\bsl H}\bar{f}(\mb{x},Ma(y_{n+1}^{-1}))y_{n+1}^n\,d\lambda_\Delta(\mb{x})\,d\mu_H(M)\,d\lambda(y_{n+1})\notag\\
&=\frac{1}{\zeta(n+1)}\int_1^{\theta^{-1}}\int_{\TDn\times\Gamma_H\bsl H}\bar{f}(\mb{x},Ma(y))y^{-(n+2)}\,d\lambda_\Delta(\mb{x})\,d\mu_H(M)\,d\lambda(y).\label{Eq:int3}
\end{align}
This shows that as $\epsilon\rar0$, the ratio of \eqref{Eq:int2} to $\lambda(\mcB_\epsilon^n)$ approaches \eqref{Eq:int3} uniformly in $\theta\in(0,1)$.

Next, we examine \eqref{Eq:Fint1}. For $\mb{x}\in\FF_\Delta^{\theta,\epsilon_\delta}(Q)$, there exists $\mb{r}\in\FF_\Delta^\theta(Q)$ such that $\|\mb{x}-\mb{r}\|\leq\epsilon_\delta/Q^{(n+1)/n}$, and thus
\[d(h(\mb{x})a(Q),h(\mb{r})a(Q))=d(a(Q)h(Q^{(n+1)/n}\mb{x}),a(Q)h(Q^{(n+1)/n}\mb{x}))\leq Q^{(n+1)/n}\|\mb{x}-\mb{r}\|\leq\epsilon_\delta.\]
So for any $\mb{r}\in\FF_\Delta^\theta(Q)$, we have
\begin{align*}
&\frac{1}{\lambda(\mcB_{\epsilon_\delta}^n)}\int\limits_{\|\mb{x}-\mb{r}\|\leq\epsilon_\delta/Q^{(n+1)/n}}|f(\mb{x},\Delta h(\mb{x})a(Q))-f(\mb{r},\Delta h(\mb{r})a(Q))|\,d\lambda_\Delta(\mb{x})\\
&\qquad\qquad\qquad\leq\frac{\delta\lambda_\Delta\big(\mcB_{\epsilon_\delta/Q^{(n+1)/n}}^n\big)}{\lambda(\mcB_{\epsilon_\delta}^n)}=\frac{\delta}{[\Z^n:\Lambda_\Delta]Q^{n+1}}.
\end{align*}
Summing this inequality over all $\mb{r}\in\FF_\Delta^\theta(Q)$, we see that the ratio of the integral in the limit \eqref{Eq:Fint1} to $\lambda(\mcB_\epsilon^n)$ is within $\delta(\#\FFDl^\theta(Q))/[\Z^n:\Lambda_\Delta]Q^{n+1}$ of the sum
\[\frac{1}{[\Z^n:\Lambda_\Delta]Q^{n+1}}\sum_{\mb{r}\in\FFDl^\theta(Q)}f(\mb{r},\Delta h(\mb{r})a(Q))).\]
Recall that $\sigma_Q$ defined in \eqref{Eq:sigma} is the asymptotic growth rate of $\#\FF(Q)$, hence $[\Z^n:\Lambda_\Delta]\sigma_Q$ is the asymptotic growth rate of $\#\FFD(Q)$. This implies that the ratio $(\#\FFDl^\theta(Q))/[\Z^n:\Lambda_\Delta]Q^{n+1}$ has an upper bound that is uniform in $Q$ (and $\theta$). Therefore, for any $\theta\in(0,1)$, we have
\begin{align}
&\lim_{Q\rar\infty}\frac{(n+1)\zeta(n+1)}{[\Z^n:\Lambda_\Delta]Q^{n+1}}\sum_{\mb{r}\in\FFDl^\theta(Q)}f(\mb{r},h(\mb{r})a(Q)))\notag\\
&=\int_1^{\theta^{-1}}\int_{\TDn\times\Gamma_H\bsl H}\bar{f}(\mb{x},Ma(y))\frac{(n+1)}{y^{n+2}}\,d\lambda_\Delta(\mb{x})\,d\mu_H(M)\,d\lambda(y).\label{Eq:Equidtheta}
\end{align}
Since $\#(\FFD(Q)\bsl\FFDl^\theta(Q))\sim\#(\FFD(\lfloor\theta Q\rfloor))\sim[\Z^n:\Lambda_\Delta]\sigma_{\theta Q}$, we may let $\theta\rar0$ to get 
\begin{align*}
\lim_{Q\rar\infty}\frac{(n+1)\zeta(n+1)}{[\Z^n:\Lambda_\Delta]Q^{n+1}}\sum_{\mb{r}\in\FFD(Q)}f(\mb{r},h(\mb{r})a(Q)))&=\int_{\TDn\times\Omega}\bar{f}(\mb{x},M)\,d\lambda_\Delta(\mb{x})\,d\mu_\Omega(M)\\
&=\int_{\TDn\times\Omega_\Delta} f(\mb{x},M)\,d\lambda_\Delta(\mb{x})\,d\mu_{\Omega_\Delta}(M).
\end{align*}

\ 

\noindent\textbf{Part \eqref{T:Equid2pt2}.}
We first show that for each $\gamma\in\Gamma$, the set $\Delta\bsl\Delta\gamma H_a$ is a connected component of $\Omega_\Delta$. The former set is clearly connected since it is the image of the connected subset $\gamma H_a\subseteq G$ under the projection $G\rar\Delta\bsl G$. Next, we show that $\Delta\bsl\Delta\gamma H_a$ is closed $\Delta\bsl G$, and hence is closed in $\Omega_\Delta$. So let $(\Delta\gamma h_ja(y_j))_j\subseteq\Delta\bsl\Delta\gamma H_a$ be a sequence (with $h_j\in H$ and $y_j\geq1$) such that $\Delta\gamma h_ja(y_j)$ converges to an element $\Delta g\in\Delta\bsl G$ as $j\rar\infty$. Then there exists a sequence $(\delta_j)_j\subseteq\Delta$ such that $\delta_j\gamma h_ja(y_j)\rar g$ as $j\rar\infty$. Define the sequence $(\gamma_j)_j\subseteq\Gamma$ by $\gamma_j=\delta_j\gamma$. We then have $\lim_{j\rar\infty}a(y_j^{-1})h_j^{-1}\gamma_j^{-1}=g^{-1}$.

The final row of $a(y_j^{-1})h_j^{-1}\gamma_j^{-1}$ is $y_j$ times the final row of $\gamma_j^{-1}$, which is an element in $\hat{\Z}^{n+1}$. Thus $(y_j)_j$ cannot be unbounded, lest the length of the final rows of the elements in  $(a(y_j^{-1})h_j^{-1}\gamma_j^{-1})_j$ be unbounded, contradicting the convergence of the matrix to $g^{-1}$. Thus $y_j$ is bounded, and by taking a subsequence, we may assume that $\lim_{j\rar\infty}y_j=y\geq1$. Then $\lim_{j\rar\infty}h_j^{-1}\gamma_j^{-1}=a(y)g^{-1}$, and as a result the final row of $h_j^{-1}\gamma_j^{-1}$, which is that of $\gamma_j^{-1}$, is eventually constant for large enough $j$. Let $\beta$ be a fixed element in $(\gamma_j)_j$ such that $\beta^{-1}$ has the same last row as $\gamma_j^{-1}$ for all large $j$. Then $\gamma_j^{-1}\beta$, and hence $\beta^{-1}\gamma_j$, is in $\Gamma_H$ for large $j$. Therefore, our original sequence $\Delta\gamma h_ja(y_j)$ can be rewritten as $\Delta\beta(\beta^{-1}\gamma_jh_j)a(y_j)=\Delta\gamma(\beta^{-1}\gamma_jh_j)a(y_j)$ (recall that $\beta\in(\delta_j\gamma)_j$). From the above, we see that $(\beta^{-1}\gamma_jh_j)_j$ is a sequence in $H$ converging to $\beta^{-1}ga(y^{-1})$, which must then be in $H$ since $H$ is closed in $G$. We also have that $y_j\rar y$, and therefore $\Delta g=\Delta\gamma(\beta^{-1}ga(y^{-1}))a(y)$ is an element of $\Delta\gamma H_a$. This proves that $\Delta\bsl\Delta\gamma H_a$ is closed in $\Delta\bsl G$.

To complete the proof that $\Delta\bsl\Delta\gamma H_a$ is a connected component of $\Omega_\Delta$, we show that for any two $\gamma_1,\gamma_2\in\Gamma$, either $\Delta\bsl\Delta\gamma_1H_a$ and $\Delta\bsl\Delta\gamma_2H_a$ are the same sets, or they are disjoint. Suppose that $\Delta\bsl\Delta\gamma_1H_a$ and $\Delta\bsl\Delta\gamma_2H_a$ are not disjoint, so there exists $\delta_1\in\Delta$, $h_1,h_2\in H$, and $y>0$ such that $\delta_1\gamma_1h_1=\gamma_2h_2a(y)$. Then $\gamma_2^{-1}\delta_1\gamma_1=h_2a(y)h_1^{-1}$. The left side of this equation is in $\Gamma$, and on the other hand, the bottom row of the right side is $(\mb{0},y^{-1})$. These facts imply that $y=1$. We then have $\gamma_2^{-1}\delta_1\gamma_1=h_2h_1^{-1}$, which is in $\Gamma_H$ since the left is in $\Gamma$ and the right is in $H$. So we have $\Delta\gamma_1H_a=\Delta\gamma_2(h_2h_1^{-1})H_a=\Delta\gamma_2H_a$, implying that $\Delta\bsl\Delta\gamma_1H_a=\Delta\bsl\Delta\gamma_2H_a$. We have thus proven that the connected components of $\Omega_\Delta$ are of the form $\Delta\bsl\Delta\gamma H_a$, with $\gamma\in\Gamma$. Furthermore, the above shows that for $\gamma_1,\gamma_2\in\Gamma$,  $\Delta\bsl\Delta\gamma_1H_a=\Delta\bsl\Delta\gamma_2H_a$ exactly when $\Delta\gamma_1=\Delta\gamma_2\tilde{\gamma}$ for some $\tilde{\gamma}\in\Gamma_H$.

Next, we show that for every $\mb{r}\in\FF_\Delta(Q)$, we have $\mb{r}\in\FF_\bA(Q)$ if and only if $\Delta h(\mb{r})a(Q)\in\Omega_\bA$. See \cite[Remark 3.3]{M1} for the observation that the full set of Farey points embed in $\Gamma\bsl\Gamma H_a$. To prove the forward implication, let $\mb{p}/q\in\FFA(Q)$. Then let $\gamma\in\Gamma$ such that $(\mb{p},q)\gamma=(\mb{0},1)$. By the definition of $\bA^*$, we have $\Delta\gamma\in\bA^*$. Now observe that
\begin{align}
\Delta h\left(\frac{\mb{p}}{q}\right)a(Q)&=\Delta\begin{pmatrix}I_n&\mb{0}^t\\-\mb{p}/q&1\end{pmatrix}
\begin{pmatrix}q^{1/n}I_n&\mb{0}^t\\\mb{0}&q^{-1}\end{pmatrix}a\left(\frac{Q}{q}\right)\notag\\
&=\Delta\gamma\left[\begin{pmatrix}q^{-1/n}I_n&\mb{0}^t\\\mb{p}&q\end{pmatrix}\gamma\right]^{-1}a\left(\frac{Q}{q}\right),\label{Eq:FsubH}
\end{align}
which is in $\Delta\gamma H_a$ since $(\mb{p},q)\gamma=(\mb{0},1)$. Thus $\Delta h(\mb{p}/q)a(Q)\in\Omega_\bA$.

To prove the converse, suppose that $\mb{p}/q\in\FFD(Q)$ satisfies $\Delta h(\mb{p}/q)a(Q)\in\Omega_\bA$, that is, there exists $\Delta\gamma\in\bA^*$ such that $\Delta h(\mb{p}/q)a(Q)\in\Delta\bsl\Delta\gamma H_a$. Repeating the process above, we find that for any $\gamma'\in\Gamma$ such that $(\mb{p},q)\gamma'=(\mb{0},1)$, we have $\Delta h(\mb{p}/q)a(Q)\in\Delta\bsl\Delta\gamma'H_a$ as well. Hence, $\Delta\bsl\Delta\gamma H_a=\Delta\bsl\Delta\gamma'H_a$, which implies that $\Delta\gamma'=\Delta\gamma\tilde{\gamma}$ for some $\tilde{\gamma}\in\Gamma_H$. We showed in the remarks before Theorem \ref{T:Equid2} that $\bA^*$ is closed under right multiplication by elements $\Gamma_H$. We can therefore conclude that $\Delta\gamma'\in\bA^*$. By the definition of $\bA^*$, we must have $(\mb{p},q)\in\bA$ and $\mb{p}/q\in\FFA(Q)$. This completes the proof that $\mb{r}\in\FF_\bA(Q)$ if and only if $\Delta h(\mb{r})a(Q)\in\Omega_\bA$.

To finish the proof of part \eqref{T:Equid2pt2}, we now apply part \eqref{T:Equid2pt1} to the product of $f$ with $\chi_{\bA^*}$, the characteristic function on $\TDn\times\Omega_\bA$. We can apply part \eqref{T:Equid2pt1} to this function since $\Omega_\bA$ is a union of connected components of $\Omega_\Delta$, and hence $\chi_{\bA^*}$ is a continuous function. We thus have
\begin{align*}
\lim_{Q\rar\infty}\frac{1}{\#\FF_\Delta(Q)}\sum_{\mb{r}\in\FFD(Q)}(\chi_{\bA^*}\cdot f)(\mb{r},\Delta h(\mb{r})a(Q))&=\int_{\T_\Delta^n\times\Omega_\Delta}(\chi_{\bA^*}\cdot f)(\mb{x},M)\,d\lambda_\Delta(\mb{x})\,d\mu_{\Omega_\Delta}(M)\\
&=\int_{\T_\Delta^n\times\Omega_\bA}f(\mb{x},M)\,d\lambda_\Delta(\mb{x})\,d\mu_{\Omega_\Delta}(M).
\end{align*}
Then since a given $\mb{r}\in\FF_\Delta(Q)$ is in $\FF_\bA(Q)$ if and only if $\Delta h(\mb{r})a(Q)\in\Omega_\bA$, we have
\[\sum_{\mb{r}\in\FFD(Q)}(\chi_{\bA^*}\cdot f)(\mb{r},\Delta h(\mb{r})a(Q))=\sum_{\mb{r}\in\FFA(Q)}f(\mb{r},\Delta h(\mb{r})a(Q)).\]
The proof of Theorem \ref{T:Equid2} is now complete.

\end{proof}

\section{Spacing statistics and the Erd\H{o}s-Sz\"{u}sz-Tur\'{a}n and Kesten distributions}\label{Sec:Proofs}

We now obtain Theorem \ref{T:Stats} as a corollary of Theorems \ref{T:Equid1} and \ref{T:Equid2} in essentially the same way Marklof obtained analogous results for the full Farey sequence in \cite{M2}. We then obtain Theorem \ref{T:KEST} as a corollary to Theorem \ref{T:Equid1} in a similar manner.

\subsection{Theorem \ref{T:Stats}}
First notice that applying Theorem \ref{T:Equid2}\eqref{T:Equid2pt2} to the constant function $f(\mb{x},M)=1$ yields
\[\lim_{Q\rar\infty}\frac{\#\FFA(Q)}{\#\FFD(Q)}=\mu_{\Omega_\Delta}(\Omega_\bA)=\frac{\#\bA^*}{[\Gamma:\Delta]}.\]
So the asymptotic growth rate of $\#\FFA(Q)$ is given by
\[\sigma_{\bA,Q}=\frac{(\#\bA^*)[\Z^{n+1}:\Lambda_\Delta]}{[\Gamma:\Delta](n+1)\zeta(n+1)}Q^{n+1}\]
Thus in the definitions of $P_Q^{\bA}$ and $P_{0,Q}^\bA$ in \eqref{Eq:Astat1} and \eqref{Eq:Astat2}, respectively, we may replace the scaling factor $(\#\FFA(Q))^{-1/n}$ by $\sigma_{\bA,Q}^{-1/n}$ without affecting the existence or value of the limits of $P_Q^\bA$ or $P_{0,Q}^\bA$. Also, notice that if $\mcD\subseteq\TDn$ has boundary of measure zero and nonempty interior, then we can apply Theorem \ref{T:Equid2}\eqref{T:Equid2pt2} to the characteristic function of $\mcD\times\Omega_\Delta$, which yields
\begin{align*}
\lim_{Q\rar\infty}\frac{\#(\FFA(Q)\cap\mcD)}{\#\FFD(Q)}&=\frac{\#\bA^*}{[\Gamma:\Delta]}\lambda_\Delta(\mcD),\quad\text{and hence}\\
\lim_{Q\rar\infty}\frac{\#(\FFA(Q)\cap\mcD)}{\#\FFA(Q)}&=\lambda_\Delta(\mcD).
\end{align*}
This proves the equidistribution of $(\FFA(Q))_Q$ in $\TDn$.

Now fix $k\in\Z_{\geq0}$ and subsets $\mcD\subseteq\TDn$ and $\mcA\subseteq\R^n$ with boundaries of measure zero and nonempty interiors. Then define the set
\[\mfC(\mcA)=\{(\mb{x},y)\in\R^n\times(0,1]:\mb{x}\in\sigma_{\bA,1}^{-1/n}y\mcA\}\subseteq\R^{n+1}.\]
For any $(\mb{p},q)\in\R^{n+1}$, we have $\mb{p}/q\in\mb{x}+\sigma_{\bA,Q}^{-1/n}\mcA$ with $1\leq q\leq Q$ if and only if
\[(\mb{p},q)h(\mb{x})a(Q)\in\mfC(\mcA).\]
So if $Q$ is large enough so that $\sigma_{\bA,Q}^{-1/n}\mcA$ fits inside a single fundamental domain for $\TDn$, we have, for $\mb{x}\in\TDn$,
\[\#((\mb{x}+\sigma_{\bA,Q}^{-1/n}\mcA)\cap\FFA(Q))=\#(\bA h(\mb{x})a(Q)\cap\mfC(\mcA)).\]
We can thus rewrite \eqref{Eq:Astat1} and \eqref{Eq:Astat2} as follows:
\begin{align}
P_Q^\bA(k,\mcD,\mcA)&=\frac{\lambda_\Delta(\{\mb{x}\in\mcD:\#(\bA h(\mb{x})a(Q)\cap\mfC(\mcA))=k\})}{\lambda_\Delta(\mcD)},\notag\\
P_{0,Q}^\bA(k,\mcD,\mcA)&=\frac{\#\{\mb{r}\in\FFA(Q)\cap\mcD:\#(\bA h(\mb{r})a(Q)\cap\mfC(\mcA))=k\}}{\#(\FFA(Q)\cap\mcD)}.\label{Eq:Astat2alt}
\end{align}
Notice that since the boundary of $\mcA$ is of measure zero in $\R^n$, the boundary of $\mfC(\mcA)$ is of measure zero in $\R^{n+1}$. Also, the characteristic function $\chi_{k,\mcD,\mcA}^\bA:\TDn\times\Delta\bsl G\rar\R$ of the set $\mcD\times\{\Delta M\in\Delta\bsl G:\#(\bA M\cap\mfC(\mcA))=k\}$ (whose second factor is a well defined subset of $\Delta\bsl G$ since $\bA$ is closed under right multiplication by $\Delta$) is discontinuous at $(\mb{x},\Delta M)$ only if $\mb{x}$ is in the boundary of $\mcD$, or if there exists $\mb{a}\in\bA$ such that $\mb{a} M$ is in the boundary of $\mfC(\mcA)$. Such $(\mb{x},\Delta M)$ comprise a set of measure zero in $\TDn\times\Delta\bsl G$. We may therefore apply Theorem \ref{T:Equid1} to $\chi_{k,\mcD,\mcA}^\bA$ to obtain
\begin{align*}
\lim_{Q\rar\infty}P_Q^\bA(k,\mcD,\mcA)&=\lim_{Q\rar\infty}\frac{1}{\lambda_\Delta(\mcD)}\int_\TDn\chi_{k,\mcD,\mcA}^\bA(\mb{x},\Delta h(\mb{x})a(Q))\,d\lambda_\Delta(\mb{x})\\
&=\frac{1}{\lambda_\Delta(\mcD)}\int_{\TDn\times\Delta\bsl G}\chi_{k,\mcD,\mcA}^\bA(\mb{x},M)\,d\lambda_\Delta(\mb{x})\,d\mu_\Delta(M)\\
&=\mu_\Delta(\{M\in\Delta\bsl G:\#(\bA M\cap\mfC(\mcA))=k\}),
\end{align*}
proving the existence of the limiting measure $P_Q^\bA(k,\mcD,\mcA)$ as $Q\rar\infty$.

Note furthermore that the limit of the expected value
\[\mathbb{E}P_Q^\bA(\mcD,\mcA)=\sum_{k=0}^\infty kP_Q^\bA(k,\mcD,\mcA)\]
as $Q\rar\infty$ exists by an application of Theorem \ref{T:Equid1} to the function
\[(\mb{x},\Delta M)\mapsto\chi_{\mcD}(\mb{x})(\#(\bA M\cap\mfC(\mcA))):\TDn\times\Delta\bsl G\rar\R,\]
where $\chi_\mcD:\TDn\rar\R$ is the characteristic function on $\mcD$. The limiting value of $\mathbb{E}P_Q^\bA(\mcD,\mcA)$ is computed as follows: First, the application of Theorem \ref{T:Equid1} yields
\begin{align}
\lim_{Q\rar\infty}\mathbb{E}P_Q^\bA(\mcD,\mcA)&=\lim_{Q\rar\infty}\frac{1}{\lambda_\Delta(\mcD)}\int_{\TDn}\chi_\mcD(\mb{x})(\#(\bA h(\mb{x})a(Q)\cap\mfC(\mcA)))\,d\lambda_\Delta(\mb{x})\notag\\
&=\frac{1}{\lambda_\Delta(\mcD)}\int_{\TDn\times\Delta\bsl G}\chi_\mcD(\mb{x})(\#(\bA M\cap\mfC(\mcA)))\,d\mu_\Delta(M)\,d\lambda_\Delta(\mb{x})\notag\\
&=\int_{\Delta\bsl G}(\#(\bA M\cap\mfC(\mcA)))\,d\mu_\Delta(M)\notag\\
&=\sum_{j=1}^J\int_{\Delta\bsl G}\sum_{\mb{a}\in\mb{a}_j\Delta}\chi_{\mfC(\mcA)}(\mb{a}M)\,d\mu_\Delta(M),\label{Eq:Expcalc1}
\end{align}
where $\chi_{\mfC(\mcA)}:\R^{n+1}\rar\R$ is the characteristic function on $\mfC(\mcA)$. Now for each $j\in\{1,\ldots,J\}$, let $\gamma_j\in\Gamma$ be an element such that $(\mb{0},1)\gamma_j=\mb{a}_j$. If we then define $\Delta_j=\Delta\cap(\gamma_j^{-1}\Gamma_H\gamma_j)$, it is readily seen that
\[\Delta_j\delta\mapsto\mb{a}_j\delta:\Delta_j\bsl\Delta\rar\mb{a}_j\Delta\]
is a bijection. Thus the term in \eqref{Eq:Expcalc1} corresponding to index $j$ is equal to
\begin{align*}
\int_{\Delta\bsl G}\sum_{\Delta_j\delta\in\Delta_j\bsl\Delta}\chi_{\mfC(\mcA)}(\mb{a}_j\delta M)\,d\mu_\Delta(M)=\int_{\Delta_j\bsl G}\chi_{\mfC(\mcA)}(\mb{a}_jM)\frac{d\mu(M)}{[\Gamma:\Delta]}.
\end{align*}
Making the change of variable $M'=\gamma_jM\gamma_j^{-1}$ and letting $\Delta_j'=\gamma_j\Delta_j\gamma_j^{-1}=\Gamma_H\cap(\gamma_j\Delta\gamma_j^{-1})$ yields
\begin{align*}
\int_{\gamma_j\Delta_j\gamma_j^{-1}\bsl G}\chi_{\mfC(\mcA)}(\mb{a}_j\gamma_j^{-1}M'\gamma_j)\,d\mu_\Delta(M')&=\int_{\Delta_j'\bsl G}\chi_{\mfC(\mcA)}((\mb{0},1)M')\frac{d\mu(M')}{[\Gamma:\Delta]}\\
&=\int_{\Gamma_H\bsl G}\sum_{\Delta_j'\gamma\in\Delta_j'\bsl\Gamma_H}\chi_{\mfC(\mcA)}((\mb{0},1)\gamma M)\frac{d\mu(M)}{[\Gamma:\Delta]}\\
&=[\Gamma_H:\Delta_j']\int_{\Gamma_H\bsl G}\chi_{\mfC(\mcA)}((\mb{0},1)M)\frac{d\mu(M)}{[\Gamma:\Delta]}.
\end{align*}
(Note that in the first equality, we use the invariance of $\mu$ to remove the factor of $\gamma_j$ on the right of $M'$.) Using the coordinates $(M,\mb{y})\mapsto MM_{\mb{y}}$ to parameterize $\Gamma_H\bsl G$ as in the proof of Theorem \ref{T:Equid2}\eqref{T:Equid2pt1}, we find that \eqref{Eq:Expcalc1} equals
\begin{align*}
\sum_{j=1}^J\frac{[\Gamma_H:\Delta_j']}{[\Gamma:\Delta]}\int_{\Gamma_H\bsl H\times \mfC(\mcA)}\frac{d\mu_H(M)\,d\lambda(\mb{y})}{\zeta(n+1)}=\frac{\lambda_\Delta(\mcA)}{\#\bA^*}\sum_{j=1}^J[\Gamma_H:\Delta_j'].
\end{align*}
Now notice that for every $j$, the set of cosets $\Delta_j'\bsl\Gamma_H$ is in one-to-one correspondence with the orbit of $\Delta\gamma_j^{-1}$ in the coset space $\Delta\bsl\Gamma$ under the action by right multiplication by $\Gamma_H$. Furthermore, we showed in our remarks preceding Theorem \ref{T:Equid2} that the set $\bA^*$ is the union of the $\Gamma_H$-orbits of the cosets $\Delta\gamma_1^{-1},\ldots,\Delta\gamma_J^{-1}$. So the sum of the indices $[\Gamma_H:\Delta_j']$ equals $\#\bA^*$, and thus
\[\lim_{Q\rar\infty}\mathbb{E}P_Q^\bA(\mcD,\mcA)=\lambda_\Delta(\mcA).\]
The result is $\lambda_\Delta(\mcA)$ instead of $\lambda(\mcA)$ due to the scaling factor $(\#\FFA(Q))^{-1/n}$ in \eqref{Eq:Astat1}; because $\#\FFA(Q)$ counts the number of points in an entire fundamental domain of $\R^n/\Lambda_\Delta$ instead of a region of $\lambda$-measure $1$, the measure of the test set $\mcA$ is scaled to compare with the measure of the fundamental domain.

Next, recall from the proof of Theorem \ref{T:Equid2}\eqref{T:Equid2pt2} that for a given $\mb{p}/q\in\FFD(Q)$, $\Delta h(\mb{p}/q)a(Q)=\Delta h(\mb{p}/q)a(q)a(Q/q)$, where $\Delta h(\mb{p}/q)a(q)\in\Delta\bsl\Gamma H$. So the points $\mb{p}/q\in\FFA(Q)\cap\mcD$ which belong to the set in the numerator of \eqref{Eq:Astat2alt} are those whose corresponding point $\Delta h(\mb{p}/q)a(Q)$ belongs to the set 
\[\Omega_{k,\mcA}^\bA=\{\Delta\gamma Ma(y)\in\Omega_\Delta:\gamma\in\Gamma, M\in H,y\geq1,\#(\bA\gamma Ma(y)\cap\mfC(\mcA))=k\}\subseteq\Omega_\Delta.\]
Let $\chi_{k,\mcD,\mcA}^{\bA,\Omega}:\Omega_\Delta\rar\R$ be the characteristic function of $\mcD\times\Omega_{k,\mcA}^\bA$. Then $\chi_{k,\mcD,\mcA}^{\bA,\Omega}$ is discontinuous at $(\mb{x},\Delta\gamma Ma(y))\in\TDn\times\Omega_\Delta$ only when $\mb{x}$ is on the boundary of $\mcD$, or when there exists a point $\mb{a}\in\hat{\Z}^{n+1}$ such that $\mb{a}Ma(y)$ is on the boundary of $\mfC(\mcA)$. Since the boundaries of $\mcD$ and $\mcA$ are of measure zero in their respective spaces, the points of discontinuity of $\chi_{k,\mcD,\mcA}^{\bA,\Omega}$ are of measure zero in $\TDn\times\Omega_\Delta$. So we can apply Theorem \ref{T:Equid2}\eqref{T:Equid2pt2} to $\chi_{k,\mcD,\mcA}^{\bA,\Omega}$, which yields
\begin{align*}
&\lim_{Q\rar\infty}\frac{\#\{\mb{r}\in\FFA(Q)\cap\mcD:\#(\bA h(\mb{r})a(Q)\cap\mfC(\mcA))=k\}}{\#\FFD(Q)}\\
&\qquad\qquad\qquad\qquad\qquad=\lambda_\Delta(\mcD)\cdot\mu_{\Omega_\Delta}(\{M\in\Omega_\bA:\#(\bA M\cap\mfC(\mcA))=k\}).
\end{align*}
Multiplying by $\#\FFD(Q)/\#(\FFA(Q)\cap\mcD)$ then yields
\begin{align*}
&\lim_{Q\rar\infty}\frac{\#\{\mb{r}\in\FFA(Q)\cap\mcD:\#(\bA h(\mb{r})a(Q)\cap\mfC(\mcA))=k\}}{\#(\FFA(Q)\cap\mcD)}\\
&\qquad\qquad\qquad\qquad\qquad\qquad=\frac{\mu_{\Omega_\Delta}(\{M\in\Omega_\bA:\#(\bA M\cap\mfC(\mcA))=k\})}{\mu_{\Omega_\Delta}(\Omega_\bA)}.
\end{align*}
This completes the proof of Theorem \ref{T:Stats}.

\subsection{Theorem \ref{T:KEST}}
Fix $A>0$, $c>1$, $k\in\Z_{\geq0}$, and $\mcD\subseteq\TDn$ with boundary of measure zero and nonempty interior. Recall that we defined the functions $\EST_{A,c,Q}^\bA$ and $K_{A,Q}^\bA$ on $\TDn$ such that for $\mb{x}\in\TDn$, $\EST_{A,c,Q}^\bA(\mb{x})$ is the number of solutions $(\mb{p},q)\in\bA$ of
\begin{equation}\label{Eq:ESTcond}
\|q\mb{x}-\mb{p}\|\leq Aq^{-1/n},\qquad Q\leq q\leq cQ,
\end{equation}
and $K_{A,Q}^\bA(\mb{x})$ is the number of solutions $(\mb{p},q)\in\bA$ of
\begin{equation}\label{Eq:Kcond}
\|q\mb{x}-\mb{p}\|\leq AQ^{-1/n},\qquad 1\leq q\leq Q.
\end{equation}
We aim to show that the limiting distribution of both functions exist. Now define the sets
\begin{align*}
\mcE_{A,c}&=\{(\mb{x},y)\in\R^n\times\R:|y|^{1/n}\|\mb{x}\|\leq A,1\leq y\leq c\},\\
\mcK_A&=\{(\mb{x},y)\in\R^n\times\R:\|\mb{x}\|\leq A, 0\leq y\leq 1\},
\end{align*}
and notice that a given $\mb{x}\in\R^n$ and $(\mb{p},q)\in\hat{\Z}^{n+1}$ satisfy \eqref{Eq:ESTcond} if and only if
\[(\mb{p},q)h(\mb{x})a(Q)\in\mcE_{A,c};\]
and similarly $\mb{x}$ and $(\mb{p},q)$ satisfy \eqref{Eq:Kcond} if and only if
\[(\mb{p},q)h(\mb{x})a(Q)\in\mcK_{A}.\]
We may therefore write $\EST_{A,c,Q}^\bA(\mb{x})$ and $K_{A,Q}^\bA(\mb{x})$ as
\begin{align*}
\EST_{A,c,Q}^\bA(\mb{x})=\#(\bA h(\mb{x})a(Q)\cap\mcE_{A,c})\quad\text{and}\quad
K_{A,Q}^\bA(\mb{x})=\#(\bA h(\mb{x})a(Q)\cap\mcK_A).
\end{align*}
Thus, if we define the functions $\eta_{A,c},\kappa_A:\Delta\bsl G\rar\Z_{\geq0}$ by
\[\eta_{A,c}(\Delta M)=\#(\bA M\cap\mcE_{A,c})\quad\text{and}\quad\kappa_A(\Delta M)=\#(\bA M\cap\mcK_A),\]
we have
\[\EST_{A,c,Q}^\bA(\mb{x})=\eta_{A,c}(\Delta h(\mb{x})a(Q))\quad\text{and}\quad
K_A^\bA(\mb{x})=\kappa_A(\Delta h(\mb{x})a(Q)).\]
Since the sets $\mcE_{A,c}$ and $\mcK_A$ have boundaries of measure zero in $\R^{n+1}$, the level sets $\eta_{A,c}^{-1}(k)$ and $\kappa_A^{-1}(k)$ have boundaries of measure zero in $\Delta\bsl G$. So we may apply Theorem \ref{T:Equid1} to the characteristic function of $\mcD\times\eta_{A,c}^{-1}(k)$, which then yields
\[\lim_{Q\rar\infty}\lambda_\Delta(\{\mb{x}\in\mcD:\#(\bA h(\mb{x})a(Q)\cap\mcE_{A,c})=k\})=\lambda_\Delta(\mcD)\mu_\Delta(\{M\in\Delta\bsl G:\#(\bA M\cap\mcE_{A,c})=k\}).\]
This is equivalent to
\[\lim_{Q\rar\infty}\frac{\lambda_\Delta(\{\mb{x}\in\mcD:\EST_{A,c,Q}^\bA(\mb{x})=k\})}{\lambda_\Delta(\mcD)}=\mu_\Delta(\{M\in\Delta\bsl G:\#(\bA M\cap\mcE_{A,c})=k\}).\]
Similarly, we can apply Theorem \ref{T:Equid1} to the characteristic function of $\mcD\times\kappa_A^{-1}(k)$ to find that
\[\lim_{Q\rar\infty}\lambda_\Delta(\{\mb{x}\in\mcD:\#(\bA h(\mb{x})a(Q)\cap\mcK_A)=k\})=\lambda_\Delta(\mcD)\mu_\Delta(\{M\in\Delta\bsl G:\#(\bA M\cap\mcK_A)=k\}),\]
and hence
\[\lim_{Q\rar\infty}\frac{\lambda_\Delta(\{\mb{x}\in\mcD:K_{A,Q}^\bA(\mb{x})=k\})}{\lambda_\Delta(\mcD)}=\mu_\Delta(\{M\in\Delta\bsl G:\#(\bA M\cap\mcK_A)=k\}).\]
We note additionally that the limiting expected values of $\EST_{A,c,Q}^\bA$ and $K_{A,Q}^\bA$ as $Q\rar\infty$ can be computed as
\[\frac{(\#\bA^*)\lambda(\mcE_{A,c})}{[\Gamma:\Delta]\zeta(n+1)}\quad\text{and}\quad\frac{(\#\bA^*)\lambda(\mcK_A)}{[\Gamma:\Delta]\zeta(n+1)},\]
respectively, in a similar way as the limit of $\mathbb{E}P_Q^\bA(\mcD,\mcA)$ computed above. This concludes the proof of Theorem \ref{T:KEST}.

\section{The distribution of the Frobenius numbers of $\bA$}\label{Sec:Frob}

In this section, we complete the proof of Theorem \ref{T:Fr}. Throughout we regularly use, for a given $\mb{y}\in\R^{n+1}$, the notation $\mb{y}'$ to denote the first $n$ coordinates of $\mb{y}$ so that $\mb{y}=(\mb{y}',y_{n+1})$. We first prove the following alteration of Theorem \ref{T:Equid2}\eqref{T:Equid2pt2}, which is analogous to \cite[Theorem 7]{M1}.

\begin{theorem}\label{T:Equid3}
Let $\mcD\subseteq\{\mb{x}\in[0,1]^{n+1}:0<x_1,\ldots,x_n\leq x_{n+1}\}$ have boundary of Lebesgue measure zero, and $f:\ol{\mcD}\times\Omega_\Delta\rar\R$ be bounded and continuous. Then we have
\begin{align}
&\lim_{T\rar\infty}\frac{1}{T^{n+1}}\sum_{\mb{a}\in\bA\cap T\mcD}f\left(\frac{\mb{a}}{T},\Delta h\left(\frac{\mb{a}'}{a_{n+1}}\right)a(T)\right)\notag\\
&\qquad\qquad\qquad=\frac{1}{\zeta(n+1)}\int_{\mcD\times\Delta\bsl\bA^*H}f(\mb{y},Ma(y_{n+1}))\,d\lambda(\mb{y})\frac{d\mu_H(M)}{[\Gamma:\Delta]}.\label{Eq:Equid3}
\end{align}
\end{theorem}

\begin{proof}
First, let $g:[0,1]\times\TDn\times\Delta\bsl\Gamma H_a\rar\R$ be continuous with compact support, and for every pair of constants $b,c\in[0,1]$ with $b<c$, define $\mathcal{L}_{g,b.c},\mathcal{U}_{g,b,c}:\TDn\times\Delta\bsl G\rar\R$ by
\[\mathcal{L}_{g,b,c}(\mb{x},M)=\inf_{u\in[b,c]}g(u,\mb{x},M)\quad\text{and}\quad\mathcal{U}_{g,b,c}(\mb{x},M)=\sup_{u\in[b,c]}g(u,\mb{x},M).\]
Then for $T>0$, define $\FFDl^{b,c}(T)=\{\mb{p}/q\in\FFD(T):bT\leq q\leq cT\}$. By the proof of Theorem \ref{T:Equid2}\eqref{T:Equid2pt1} (in particular, see \eqref{Eq:Equidtheta}), we have
\begin{align*}
&\limsup_{T\rar\infty}\frac{1}{\#\FFD(T)}\sum_{\mb{p}/q\in\FFDl^{b,c}(T)}g\left(\frac{q}{T},\frac{\mb{p}}{q},\Delta h\left(\frac{\mb{p}}{q}\right)a(T)\right)\\
&\qquad\qquad\qquad\leq\lim_{T\rar\infty}\frac{1}{\#\FFD(T)}\sum_{\mb{r}\in\FFDl^{b,c}(T)}\mathcal{U}_{g,b,c}(\mb{r},\Delta h(\mb{r})a(T))\\
&\qquad\qquad\qquad=\int_b^c\int_{\TDn\times\Delta\bsl\Gamma H}\mathcal{U}_{g,b,c}(\mb{x},Ma(y))(n+1)y^n\,d\lambda_\Delta(\mb{x})\,\frac{d\mu_H(M)}{[\Gamma:\Delta]}\,d\lambda(y).
\end{align*}
So for any finite partition $0=b_0<b_1<\cdots<b_m=1$ of $[0,1]$,
\begin{align*}
&\limsup_{T\rar\infty}\frac{1}{\#\FFD(T)}\sum_{\mb{p}/q\in\FFDl(T)}g\left(\frac{q}{T},\frac{\mb{p}}{q},\Delta h\left(\frac{\mb{p}}{q}\right)a(T)\right)\\
&\qquad\qquad\qquad\leq\sum_{j=0}^{m-1}\int_{b_j}^{b_{j+1}}\int_{\TDn\times\Delta\bsl\Gamma H}\mathcal{U}_{g,b,c}(\mb{x},Ma(y))(n+1)y^n\,d\lambda_\Delta(\mb{x})\,\frac{d\mu_H(M)}{[\Gamma:\Delta]}\,d\lambda(y).
\end{align*}
By the uniform continuity of $g$, we can take the infimum over all finite partitions $\{b_j\}_j$ of $[0,1]$ to obtain
\begin{align*}
&\limsup_{T\rar\infty}\frac{1}{\#\FFD(T)}\sum_{\mb{p}/q\in\FFDl(T)}g\left(\frac{q}{T},\frac{\mb{p}}{q},\Delta h\left(\frac{\mb{p}}{q}\right)a(T)\right)\\
&\qquad\qquad\qquad\leq\int_0^1\int_{\TDn\times\Delta\bsl\Gamma H}g(y,\mb{x},Ma(y))(n+1)y^n\,d\lambda_\Delta(\mb{x})\,\frac{d\mu_H(M)}{[\Gamma:\Delta]}\,d\lambda(y).
\end{align*}
We can similarly use the functions $\mathcal{L}_{g,b,c}$ to get the reverse inequality while replacing $\limsup$ with $\liminf$; and hence
\begin{align}
&\lim_{T\rar\infty}\frac{1}{\#\FFD(T)}\sum_{\mb{p}/q\in\FFDl(T)}g\left(\frac{q}{T},\frac{\mb{p}}{q},\Delta h\left(\frac{\mb{p}}{q}\right)a(T)\right)\notag\\
&\qquad\qquad\qquad=\int_0^1\int_{\TDn\times\Delta\bsl\Gamma H}g(y,\mb{x},Ma(y))(n+1)y^n\,d\lambda_\Delta(\mb{x})\,\frac{d\mu_H(M)}{[\Gamma:\Delta]}\,d\lambda(y).\label{Eq:EquidCor}
\end{align}
By a standard approximation argument, we can replace $g$ in the above equality by the function $\tilde{g}:[0,1]\times\TDn\times\Delta\bsl\Gamma H\rar\R$ defined by
\[\tilde{g}(y,\mb{x},M)=\chi_{(0,1]^n}(\mb{x})\chi_\mcD(y\ul{\mb{x}},y)\chi_{\bA^*}(\mb{x},M)f((y\ul{\mb{x}},y),M).\]
Here $\chi_{(0,1]^n}:\TDn\rar\R$ is the characteristic function of $(0,1]^n$ viewed as a subset of $\TDn$, $\chi_\mcD:\R^{n+1}\rar\R$ is the characteristic function of $\mcD$, and $\chi_{\bA^*}$ is the characteristic function of $\TDn\times\Omega_\bA$ as defined in the proof of Theorem \ref{T:Equid2}\eqref{T:Equid2pt2}. Also, assuming $\mb{x}\in\TDn$ is such that $\chi_{(0,1]^n}(\mb{x})=1$, $\ul{\mb{x}}\in\R^n$ denotes the representative of $\mb{x}\in\TDn$ lying in $(0,1]^n$. Applying \eqref{Eq:EquidCor} with $\tilde{g}$ yields
\begin{align*}
&\lim_{T\rar\infty}\frac{(n+1)\zeta(n+1)}{[\Z^n:\Lambda_\Delta]T^{n+1}}\sum_{(\mb{p},q)\in\bA\cap T\mcD}f\left(\frac{(\mb{p},q)}{T},\Delta h\left(\frac{\mb{p}}{q}\right)a(T)\right)\\
&\qquad\qquad\qquad=\int_{[0,1]^{n+1}\times\Delta\bsl\bA^*H}\chi_\mcD(y\mb{x},y)f((y\mb{x},y),Ma(y))(n+1)y^n\,\frac{d\lambda(\mb{x})}{[\Z^n:\Lambda_\Delta]}\frac{d\mu_H(M)}{[\Gamma:\Delta]}d\lambda(y).
\end{align*}
Making the substitutions $\mb{a}=(\mb{p},q)$ and $\mb{y}=(y\mb{x},y)$ and simplifying yields \eqref{Eq:Equid3}.

\end{proof}

A way of interpreting Theorem \ref{T:Equid3} is that $\{(\mb{a}/T,\Delta h(\mb{a}'/a_{n+1})a(T)):\mb{a}\in\bA\cap T\mcD\}$ is, by \eqref{Eq:FsubH}, a subset of
\[\mathcal{M}_{\Delta,\bA,\mcD}=\{(\mb{y},Ma(y_{n+1})):\mb{y}\in\ol{\mcD},M\in\Delta\bsl\bA^* H\}\subseteq\ol{\mcD}\times\Delta\bsl\Gamma H_a\]
that equidistributes with respect to a scalar multiple of the measure $\nu_{\Delta}$ obtained as the pushforward of the measure $\lambda\times\mu_H$ on $\ol{\mcD}\times\Delta\bsl\Gamma H$ via the map
\begin{equation}\label{Eq:pfw}
(\mb{y},M)\mapsto(\mb{y},Ma(y_{n+1})):\ol{\mcD}\times\Delta\bsl\Gamma H\rar\ol{\mcD}\times\Delta\bsl\Gamma H_a.
\end{equation}
It is easy to see that $\mathcal{M}_{\Delta,\bA,\mcD}$ is a closed subset of $\ol{\mcD}\times\Delta\bsl G$, and we may therefore apply Theorem \ref{T:Equid3} to characteristic functions of appropriate subsets of $\mathcal{M}_{\Delta,\bA,\mcD}$.

A key observation in \cite{M1} is that for $\mb{a}\in\hat{\Z}_{\geq2}^{n+1}$, the quantity
\[\frac{F(\mb{a})+\sum_{j=1}^{n+1}a_j}{(a_1\cdots a_{n+1})^{1/n}}\]
can be written as an appropriate function of $\mb{a}/T$ and $h(\mb{a}'/a_{n+1})a(T)\in\Gamma\bsl\Gamma H_a$ as follows: Let $\rho:\Gamma_0\bsl G_0\rar\R$ be the covering radius of the simplex $\delta^{(n)}$ as defined in \eqref{Eq:covrad}. Then let $H^\dagger=(H^t)^{-1}=\{(M^t)^{-1}:M\in H\}$ and define the projection $\pi_0$ from $\Gamma\bsl\Gamma H^\dagger\{a(y):y>0\}$ to $\Gamma_0\bsl G_0$ so that for any $y>0$ and
\[M=\Gamma\begin{pmatrix}A&\mb{0}^t\\\mb{b}&1\end{pmatrix}\in\Gamma\bsl\Gamma H^\dagger,\]
we have $\pi_0(Ma(y))=\Gamma_0A$. Next, for $\mb{y}=(y_1,\ldots,y_n)\in\R^n$, we define the matrices
\[m(\mb{y})=\begin{pmatrix}(y_1\cdots y_n)^{-1/n}\diag(y_1,\ldots,y_n)&\mb{0}^t\\\mb{0}&1\end{pmatrix}
\quad\text{and}\quad h^\dagger(\mb{y})=\begin{pmatrix}I_n&\mb{y}^t\\\mb{0}&1\end{pmatrix},\]
where for the former, we assume that $y_j>0$ for all $j$. (Note that $h^\dagger(\mb{y})$ is the inverse transpose of $h(\mb{y})$.)  We then have
\begin{align*}
\frac{F(\mb{a})+\sum_{j=1}^{n+1}a_j}{(a_1\cdots a_{n+1})^{1/n}}&=(\rho\circ\pi_0)\left(\Gamma h^\dagger\left(\frac{\mb{a}'}{a_{n+1}}\right)a(T^{-1})m\left(\frac{\mb{a}'}{T}\right)\right)\\
&=(\rho\circ\pi_0)\left(\Gamma\left(\left(h\left(\frac{\mb{a}'}{a_{n+1}}\right)a(T)\right)^t\right)^{-1}m\left(\frac{\mb{a}'}{T}\right)\right).
\end{align*}
So for $R\geq0$ and a subset $\mcD\subseteq[0,1]^{n+1}$ as in Theorem \ref{T:Equid3}, one establishes the limiting value of
\begin{equation}\label{Eq:FNgR}
\frac{1}{T^{n+1}}\#\left\{\mb{a}\in\hat{\Z}_{\geq2}^{n+1}\cap T\mcD:\frac{F(\mb{a})+\sum_{j=1}^{n+1}a_j}{(a_1\cdots a_{n+1})^{1/n}}>R\right\}
\end{equation}
by applying \cite[Theorem 7]{M1}, i.e., Theorem \ref{T:Equid3} with $\Delta=\Gamma$ and $\bA=\hat{\Z}^{n+1}$, to the characteristic function of the subset
\begin{align*}
\mcA_R&=\left\{(\mb{y},Ma(y_{n+1})):(\mb{y},M)\in\ol{\mcD}\times\Gamma\bsl\Gamma H,(\rho\circ\pi_0)(((Ma(y_{n+1}))^t)^{-1}m(\mb{y}'))>R\right\}\\
&\subseteq\mathcal{M}_{\Gamma,\hat{Z}^{n+1},\mcD}.
\end{align*}
To ensure this application is valid, Marklof showed that the boundary of $\mcA_R$ in $\mathcal{M}_{\Gamma,\hat{\Z}^{n+1},\mcD}$ is of $\nu_\Gamma$-measure zero ($\nu_\Gamma$ being the pushforward measure via \eqref{Eq:pfw} with $\Delta=\Gamma$). A brief description of his proof is as follows: If $\Psi_{n+1}:\R_{\geq0}\rar\R_{\geq0}$ is defined by
\[\Psi_{n+1}(R)=\mu_0(\{A\in\Gamma_0\bsl G_0:\rho(A)>R\}),\]
then $\nu_\Gamma(\mcA_R)=\lambda(\mcD)\Psi_{n+1}(R)$. Marklof showed that $\rho:\Gamma_0\bsl G_0\rar\R$ is continuous, implying that for any $\epsilon>0$, the measure $\nu_\Gamma(\partial\mcA_R)$ of the boundary of $\mcA_R$ is at most $\lambda(\mcD)(\Psi_{n+1}(R+\epsilon)-\Psi_{n+1}(R-\epsilon))$, and any limit point of \eqref{Eq:FNgR} as $T\rar\infty$ is in the interval
\[\left(\frac{\lambda(\mcD)}{\zeta(n+1)}\Psi_{n+1}(R-\epsilon),\frac{\lambda(\mcD)}{\zeta(n+1)}\Psi_{n+1}(R+\epsilon)\right).\]
Marklof also showed that $\mu_0(\{A\in\Gamma_0\bsl G_0:\rho(A)=R\})=0$, and thus $\Psi_{n+1}$ is continuous. This implies that $\nu_\Gamma(\partial\mcA_R)=0$ and the limit of \eqref{Eq:FNgR} as $T\rar\infty$ equals
\[\frac{\lambda(\mcD)}{\zeta(n+1)}\Psi_{n+1}(R).\]

In our situation, we apply Theorem \ref{T:Equid3} to the characteristic function of the set
\[\mcA_{\Delta,\bA,R}=\left\{(\mb{y},Ma(y_{n+1})):(\mb{y},M)\in\ol{\mcD}\times\Delta\bsl\bA^*H,(\rho\circ\pi_0)(((Ma(y_{n+1}))^t)^{-1}m(\mb{y}'))>R\right\},\]
where we view $\pi_0$ as a function on $\Delta\bsl\Gamma H^\dagger\{a(y):y>0\}$ by composition with the natural projection $\pi_\Delta$. We then have
\begin{align*}
\nu_\Delta(\mcA_{\Delta,\bA,R})&=(\lambda\times\mu_H)\left(\{(\mb{y},M)\in\ol{\mcD}\times\Delta\bsl\bA^*H:(\rho\circ\pi_0)(((Ma(y_{n+1}))^t)^{-1}m(\mb{y}'))>R\}\right)\\
&=\int_{\mcD}\mu_H(\{M\in\Delta\bsl\bA^*H:(\rho\circ\pi_0)(((Ma(y_{n+1}))^t)^{-1}m(\mb{y}'))>R\})\,d\lambda(\mb{y})\\
&=\int_{\mcD}\mu_H(\{M\in\Delta\bsl\bA^*H:(\rho\circ\pi_0)((M^t)^{-1})>R\})\,d\lambda(\mb{y})\\
&=(\#\bA^*)\lambda(\mcD)\mu_H(\{M\in\Gamma\bsl\Gamma H:(\rho\circ\pi_0)((M^t)^{-1})>R\})\\
&=(\#\bA^*)\lambda(\mcD)\Psi_{n+1}(R).
\end{align*}
For the third equality we use the $\{a(y):y>0\}$-invariance of $\pi_0$, and the $H$-invariance of $\mu_H$, and for the fourth equality we use the fact that the natural projection $\Delta\bsl\bA^*H\rar\Gamma\bsl\Gamma H$ is locally $\mu_H$-preserving and a $(\#\bA^*)$-to-one cover. Using again the continuity of $\rho$ and $\Psi_{n+1}$, we see that $\nu_\Delta(\partial\mcA_{\Delta,\bA,R})=0$, and the application of Theorem \ref{T:Equid3} to the characteristic function of $\mcA_{\Delta,\bA,R}$ yields
\[\lim_{T\rar\infty}\frac{1}{T^{n+1}}\#\left\{\mb{a}\in\hat{\Z}_{\geq2}^{n+1}\cap\bA\cap T\mcD:\frac{F(\mb{a})+\sum_{j=1}^{n+1}a_j}{(a_1\cdots a_{n+1})^{1/n}}>R\right\}=\frac{(\#\bA^*)\lambda(\mcD)}{[\Gamma:\Delta]\zeta(n+1)}\Psi_{n+1}(R).\]
One can easily remove the sum $\sum_{j=1}^{n+1}a_j$ above by replacing $\mcD$ with $\mcD\cap[\eta,1]^{n+1}$ for $\eta>0$ (under this condition the expressions $(a_1\cdots a_{n+1})^{-1/n}\sum_{j=1}^{n+1}a$ decay uniformly to zero as $T\rar\infty$), and then letting $\eta\rar0$. (See \cite[Lemma 2]{M1}.) This completes the proof of Theorem \ref{T:Fr} in the case where $\mcD\subseteq\{\mb{x}\in[0,1]^{n+1}:0<x_1,\ldots,x_n\leq x_{n+1}\}$. To remove the condition that $x_1,\ldots,x_n\leq x_{n+1}$ for every $(x_1,\ldots,x_{n+1})\in\mcD$, one can partition $\mcD$ into a disjoint regions $\mcD_1,\ldots,\mcD_{n+1}$ such that $\mcD_k\subseteq\{\mb{x}\in[0,1]^{n+1}:0<x_i\leq x_k,i\neq k\}$. For each $k$, one can find appropriate permutation matrices $P_k$ such that $\mcD_kP_k\subseteq\{\mb{x}\in[0,1]^{n+1}:0<x_1,\ldots,x_n\leq x_{n+1}\}$, and then establish the limit
\[\lim_{T\rar\infty}\frac{1}{T^{n+1}}\#\left\{\mb{a}\in\hat{\Z}_{\geq2}^{n+1}\cap\bA P_k\cap T(\mcD_kP_k):\frac{F(\mb{a})}{(a_1\cdots a_{n+1})^{1/n}}>R\right\}\]
by the above process. However, given that $\bA P_k=\bigcup_{j=1}^J\mb{a}_jP_k(P_k^{-1}\Delta P_k)$, $\Delta$ must be replaced by $P_k^{-1}\Delta P_k$. Finally, to remove the condition that $\mcD\subseteq[0,1]^{n+1}$, one can rescale the parameter $T$.

\begin{acknowledgements}
I thank Jens Marklof for his suggestion to examine the limiting distribution of Frobenius numbers for restricted sets of lattice points as another application of Theorem \ref{T:Equid2}. I also thank Nimish Shah for many helpful discussions.
\end{acknowledgements}

\end{document}